\documentclass[a4paper,12pt]{article}

\usepackage{amscd}
\usepackage{amsfonts}
\usepackage{amsmath}
\usepackage{amssymb}
\usepackage{amsthm}
\usepackage{ascmac}
\usepackage[T1]{fontenc}
\usepackage{here}
\usepackage{mathrsfs}
\usepackage{slashbox}
\usepackage{txfonts}
\usepackage[all]{xy}

\allowdisplaybreaks

\theoremstyle{plain}
\newtheorem{thm}{Theorem}[section]
\newtheorem{lmm}[thm]{Lemma}
\newtheorem{prp}[thm]{Proposition}
\newtheorem{crl}[thm]{Corollary}
\theoremstyle{definition}

\newtheorem{rmk}[thm]{Remark}
\newtheorem{exm}[thm]{Example}

\def\ens#1{\mathchoice{\left\{ #1 \right\}}{\{ #1 \}}{\{ #1 \}}{\{ #1 \}}}
\def\set#1#2{\mathchoice{\left\{ #1 \middle| #2 \right\}}{\{ #1 \mid #2 \}}{\{ #1 \mid #2 \}}{\{ #1 \mid #2 \}}}
\def\r#1{\text{\rm #1}}
\def\t#1{\text{#1}}
\def\v#1{\mathchoice{\left| #1 \right|}{| #1 |}{| #1 |}{| #1 |}}
\def\n#1{\mathchoice{\left\| #1 \right\|}{\| #1 \|}{\| #1 \|}{\| #1 \|}}

\def\ol#1{\overline{#1}{}}
\def\tl#1{\tilde{#1}{}}
\def\ul#1{\underline{#1}{}}
\def\wh#1{\widehat{#1}{}}

\newcommand{\im}{\r{im}}

\newcommand{\coim}{\r{coim}}

\newcommand{\bC}{\mathbb{C}}

\newcommand{\bF}{\mathbb{F}}

\newcommand{\bL}{\mathbb{L}}

\newcommand{\bN}{\mathbb{N}}

\newcommand{\bR}{\mathbb{R}}

\newcommand{\bZ}{\mathbb{Z}}

\newcommand{\cB}{\mathscr{B}}
\newcommand{\cC}{\mathscr{C}}
\newcommand{\cD}{\mathscr{D}}

\newcommand{\cK}{\mathscr{K}}

\newcommand{\cM}{\mathscr{M}}

\newcommand{\cT}{\mathscr{T}}
\newcommand{\cU}{\mathscr{U}}

\newcommand{\rC}{\r{C}}

\newcommand{\rH}{\r{H}}

\newcommand{\rM}{\r{M}}

\newcommand{\C}{\bC}
\newcommand{\F}{\bF}
\newcommand{\N}{\bN}

\newcommand{\R}{\bR}
\newcommand{\Z}{\bZ}

\newcommand{\Qp}{\mathbb{Q}_p}

\newcommand{\Ab}{\r{Ab}}

\newcommand{\Alg}{\r{Alg}}

\newcommand{\Ch}{\r{Ch}}

\newcommand{\Cone}{\r{Cone}}

\newcommand{\CoMod}{\r{CoMod}}
\newcommand{\Cyl}{\r{Cyl}}
\newcommand{\Desc}{\r{Desc}}

\newcommand{\Hom}{\r{Hom}}
\newcommand{\uHom}{\ul{\Hom}}

\newcommand{\id}{\r{id}}

\newcommand{\Mod}{\r{Mod}}

\newcommand{\op}{\r{op}}

\newcommand{\Res}{\r{Res}}

\newcommand{\Set}{\r{Set}}

\newcommand{\Spec}{\r{Spec}}

\newcommand{\Tot}{\r{Tot}}

\newcommand{\Cech}{$\check{\t{C}}$ech }
\newcommand{\Frechet}{Fr\'echet }
\newcommand{\GelfandNaimark}{Gel'fand--Naimark }

\newcommand{\BAb}{\Ab^{\r{A}}}
\newcommand{\BAbleq}{\Ab^{\r{A}}_{\leq 1}}
\newcommand{\CAlg}{C^* \r{-Alg}}
\newcommand{\CH}{\r{CH}}
\newcommand{\Chba}{\Ch^{-}}
\newcommand{\CO}{\r{CO}}
\newcommand{\cTd}{\cT_{\r{d}}}
\newcommand{\cTnd}{\cT_{\r{nd}}}
\newcommand{\Der}{\r{Der}}
\newcommand{\Derba}{\Der^{-}}
\newcommand{\derh}{d^{\r{h}}}
\newcommand{\derv}{d^{\r{v}}}
\newcommand{\LH}{\r{LH}}
\newcommand{\naive}{\r{naive}}
\newcommand{\NBAb}{\Ab^{\r{nA}}}
\newcommand{\NBAbleq}{\Ab^{\r{nA}}_{\leq 1}}
\newcommand{\Sol}{\r{Sol}}
\newcommand{\TDCH}{\r{TDCH}}
\newcommand{\NSet}{\r{NSet}}

\title{Homotopy Epimorphisms and Derived Tate's Acyclicity for Commutative $C^*$-algebras}
\author{Federico Bambozzi and Tomoki Mihara}
\date{}

\begin{document}

\maketitle

\begin{abstract}
We study homotopy epimorphisms and covers formulated in terms of derived Tate's acyclicity for commutative $C^*$-algebras and their non-Archimedean counterparts. We prove that a homotopy epimorphism between commutative $C^*$-algebras precisely corresponds to a closed immersion between the compact Hausdorff topological spaces associated to them, and a cover of a commutative $C^*$-algebra precisely corresponds to a topological cover of the compact Hausdorff topological space associated to it by closed immersions admitting a finite subcover. This permits us to prove derived and non-derived descent for Banach modules over commutative $C^*$-algebras.
\end{abstract}

\tableofcontents

\section{Introduction}
\label{Introduction}

Since Tate's invention of rigid geometry (cf.\ \cite{Tat71}), the geometrical study of commutative Banach rings has been an important topic in number theory. The geometrical study of commutative rings has brought various benefits such as cohomological approaches to construct Galois representations, but the geometrical study of commutative Banach rings presents more difficult challenges than the geometrical study of commutative rings. One of the main difficulties is that the additive category of Banach modules over a Banach ring is not an Abelian category. In particular, this creates the problem of the correct use of the methods of homological algebra in this context. One solution is to work in Banach modules under appropriate finiteness assumptions to obtain an Abelian category. For example, P.\ Schneider and J.\ Teitelbaum constructed in \cite{ST02} \S 3 an Abelian subcategory of the additive category of Banach representations of a compact $p$-adic Lie group by introducing the notion of admissibility, which can be characterised in terms of the finiteness over the Iwasawa algebra through their duality theory (cf.\ \cite{ST02} Theorem 2.3 and \cite{ST02} Theorem 3.5). Another solution is to work within the theory of quasi-Abelian categories and apply the extension of the derived categorical approach to a quasi-Abelian category as introduced by J.-P.\ Schniders in \cite{Sch99}. This approach has been recently developed in the series of papers (cf.\ \cite{BB16}, \cite{BK17}, \cite{BBK19}, and \cite{BK20}). Once the notions of derived categories and derived functors are correctly extended to the analytic setting, it is natural to consider the derived variant of Tate's acyclicity. In order to explain the benefit to consider derived Tate's acyclicity, we recall the historical background on Tate's acyclicity.

\vspace{0.1in}
Tate's acyclicity is a desired property of commutative Banach rings in rigid geometry, that permits to equip their associated spectra with structure sheaves. Although commutative Banach algebras over a complete valuation field satisfying a finiteness condition called {\it affinoid algebras} satisfy Tate's acyclicity, it is known that a general commutative Banach ring does not necessarily satisfy Tate's acyclicity. The lack of Tate's acyclicity was one of the biggest obstructions to develop analytic geometry for general commutative Banach rings, e.g.\ commutative Banach algebras over $\Z$ and topologically infinitely generated commutative Banach algebras over a field. Recently, P.\ Scholze invented a novel foundation of rigid geometry called {\it perfectoid theory}, in which he verified Tate's acyclicity for topologically infinitely generated commutative Banach algebras over a field called {\it perfectoid algebras}, and asked several related open questions on Tate's acyclicity. Although several affirmative answers and negative answers to Scholze's conjectures were given in \cite{Mih16} and \cite{BV18}, it is still difficult to prove or disprove Tate's acyclicity for explicit examples.

\vspace{0.1in}
On the other hand, derived Tate's acyclicity for a rational cover always hold as proved in \cite{Sch19} Proposition 13.16 and \cite{BK20} Theorem 4.15, and Tate's acyclicity for a general cover consisting of flat objects in the analytic sense implies derived Tate's acyclicity (cf.\ Proposition \ref{relation between derived covers and non-derived covers}). As evidence of the reasonability and the naturality of the formulation, we give a complete characterisation of derived Tate's acyclicity for commutative $C^*$-algebras in terms of topological covers of the corresponding compact Hausdorff topological spaces by closed immersions and show the equivalence between Tate's acyclicity and derived Tate's acyclicity for them in Theorem \ref{relation of cover by subsets} and Theorem \ref{relation of cover}.

\vspace{0.1in}
The results include the non-Archimedean counterparts, while the characterisation becomes a little more complicated. For example, the projection $[0,1] \to \ens{0}$ is not a closed immersion but induces an isometric isomorphism between the algebras of continuous functions with values in a complete valuation field $k$. Conversely, the closed immersion $\ens{0,1} \hookrightarrow [0,1]$ induces a $k$-algebra homomorphism between the algebras which cannot be part of a cover in our setting, as we will show in Remark \ref{closed immersion is not necessarily homotopy epimorphism}. Instead of a closed immersion, we introduce the notion of an $\NBAb$-embedding in section \ref{Topological covers}. For any totally disconnected compact Hausdorff topological space $X$, the notion of a topological cover of $X$ by $\NBAb$-embeddings precisely corresponds to a cover of the algebra of continuous functions $X \to k$ formulated in terms of derived Tate's acyclicity (or equivalently Tate's acyclicity).

\vspace{0.1in}
As an application, we show in Theorem \ref{derived effective descent} that a finite set of closed immersions into a compact Hausdorff topological space (resp.\ a totally disconnected compact Hausdorff topological space) $X$ is a topological cover of $X$ if and only if it satisfies effective descent for Banach modules over the algebra of continuous functions with values in $\R$ or $\C$ (resp.\ in $k$). It is quite remarkable that Tate's acyclicity is not directly used in the proof of effective descent, because the category of Banach spaces is not balanced and hence a faithful functor from it might not be conservative. Contrary, derived Tate's acyclicity plays the most important role in the proof, because the derived category is balanced and hence a faithful functor from it is conservative as we will show in Lemma \ref{Faithful functor from balanced category is conservative}. First, we show Kiehl's Theorem B for both Banach modules and objects in the derived category in Corollary \ref{Theorem B}. Using Kiehl's Theorem B for objects in the derived category and Balmer's criterion for effective descent (cf.\ \cite{Bal12} Corollary 3.1), which is only applicable to a triangulated category, we show that any descent problem for Banach modules has a solution in the derived category. Using general results by Schneiders on the left heart cohomology for a quasi-Abelian category, we show that the solution in the derived category is isomorphic to a Banach module to conclude the solvability of the descent problem within the category of Banach modules.

\vspace{0.1in}
We summarise the contents of this paper. Section \S \ref{Preliminaries} recollects standard convention and terminology. First, in \S \ref{Banach algebras}, we recall terminology on Banach modules and quasi-Abelian categories. Secondly, in \S \ref{Derived functors}, we recall the theory of derived categories and derived functors for quasi-Abelian categories, specialising the discussion for Banach modules. Thirdly, in \S \ref{Derived analytic geometry}, we recall homotopy epimorphisms and derived Tate's acyclicity for commutative Banach algebras.

\vspace{0.1in}
In \S \ref{Homotopy Zariski open immersions versus closed immersions}, we study homotopy epimorphisms between the Banach $R$-algebras $\rC(X,R)$ of continuous functions $X \to R$, for compact Hausdorff topological spaces $X$ and a commutative Banach ring $R$, and characterise homotopy epimorphisms as closed immersions in three subsections. First, in \S \ref{Trivial examples}, we show that a clopen subset of the spectrum of $\rC(X, R)$ corresponds to a homotopy epimorphism and a disjoint clopen cover satisfies derived Tate's acyclicity. Secondly, in \S \ref{Projectivity criterion}, we give a criterion for the property that the closed ideal $I_{X,R,K} \subset \rC(X,R)$, corresponding to a closed subset $K \subset X$, is projective and apply it to obtain a criterion for the property that the restriction map $\pi_{X,R,K} \colon \rC(X,R) \to \rC(K,R)$ is a homotopy epimorphism. This criterion holds for $X$ and $K$ belonging to a quite restrictive class, while it puts little restrictions on $R$. Thirdly, in \S \ref{Flatness criterion}, we give a criterion for the property that $I_{X,R,K}$ is flat and apply it to obtain another criterion for the property that $\pi_{X,R,K}$ is a homotopy epimorphism. It requires $R$ to be $\R$, $\C$, a finite field, or a local field, while $X$ and $K$ are allowed to be quite general. More precisely, we only assume that $X$ is totally disconnected when $R$ is non-Archimedean.

\vspace{0.1in}
In \S \ref{Derived covers versus topological covers}, we study derived Tate's acyclicity for the Banach $R$-algebras $\rC(X,R)$ and characterise derived Tate's acyclicity in terms of topological covers in three subsections. First, in \S \ref{Topological covers}, we introduce the notion of topological cover in a way that depends on whether $R$ is Archimedean or non-Archimedean. When $R$ is Archimedean, the notion of a topological cover of $X$ in our context is equivalent to that of a family of closed immersions into $X$ admitting a finite jointly surjective subfamily. On the other hand, when $R$ is non-Archimedean, then it is equivalent to a  family of continuous maps whose image by the Banaschewski compactification functor is a family of closed immersion admitting a finite jointly surjective subfamily. Here, Banaschewski compactification means a non-Archimedean counterpart of the Stone--\Cech compactification, i.e.\  the universal totally disconnected Hausdorff compactification. Secondly, in \S \ref{Acyclicity for a topological cover by closed immersions}, we show that, under the hypothesis stated in \S \ref{Flatness criterion}, a family of closed subsets of $X$ is a topological cover if and only if the corresponding set of homomorphisms from $\rC(X,R)$ satisfies derived Tate's acyclicity (or equivalently, Tate's acyclicity). Thirdly, in \S \ref{Acyclicity for a general topological cover}, we show that a family of continuous maps to $X$ is a topological cover if and only if the corresponding family of homomorphisms from $\rC(X,R)$ satisfies derived Tate's acyclicity (or equivalently, Tate's acyclicity) under conditions weaker than that in \S \ref{Flatness criterion}. For example, we do not impose $X$ to be totally disconnected when $R$ is non-Archimedean, as we have used Banaschewski compactification functor in the formulation of the notion of a topological cover.

\vspace{0.1in}
In \S \ref{Application to derived and non-derived descent}, we apply our main results to the problem of derived and non-derived effective descent for complexes of Banach modules over $\rC(X,R)$, under the same hypothesis on $X$ and $R$ as in the previous two sections. We prove that for a finite jointly surjective family $S$ of closed immersions into $X$, any Banach $\rC(X,R)$-module (resp.\ object of the derived category $\Derba_{\cC}(\rC(X,R))$) can be uniquely reconstructed from its restrictions to each $K \hookrightarrow X$ in $S$, and conversely any Banach $\rC(X,R)$-module (resp.\ object of $\Derba_{\cC}(\rC(X,R))$) is isomorphic to a family of Banach $\rC(K,R)$-modules (resp.\ objects of $\Derba_{\cC}(\rC(K,R))$) that are isomorphic on the intersections in a compatible way.

\section{Preliminaries}
\label{Preliminaries}

In this paper we use the following notation and conventions. If $\cC$ is a category we use the same symbol $\cC$ to denote the class of objects of $\cC$, so that $X \in \cC$ means that $X$ is an object of $\cC$. We fix a Grothendieck universe, and always implicitly use it in a standard way in order to avoid the set-theoretic problem on the localisation of a category. Monoid objects of a symmetric monoidal category are assumed to be commutative, and rings, algebras, and $C^*$-algebras are assumed to be unital and commutative.

\vspace{0.1in}
We follow the terminology of the theory of quasi-Abelian categories as developed in \cite{Sch99}. We briefly recall the basic terminology of loc.\ cit.\ used throughout the whole paper. Let $\cC$ be an additive category with all kernels and cokernels. A morphism $f$ in $\cC$ is said to be {\it strict} if the canonical morphism $\coim(f) \to \im(f)$ is an isomorphism. The additive category $\cC$ is said to be {\it quasi-Abelian} if the family of short exact sequences 
\begin{eqnarray*}
0 \to X \stackrel{f}{\to} Y \stackrel{g}{\to} Z \to 0
\end{eqnarray*}
where $f$ and $g$ are strict morphisms forms a Quillen exact structure (cf.\ \cite{Sch99} Remark 1.1.11). We say that a sequence of morphisms 
\begin{eqnarray*}
X \stackrel{f}{\to} Y \stackrel{g}{\to} Z
\end{eqnarray*}
is {\it strictly exact at $Y$} (resp.\ {\it strictly coexact at $Y$}) if $\im(f) = \ker(g)$ and $f$ is a strict morphism (resp.\ $g$ is a strict morphism). These notions are extended to longer sequences of morphisms by asking that the above conditions hold for any pair of adjacent morphisms of the sequence. For example, for a short exact sequence being strictly exact is equivalent to being strictly coexact. We refer to \S 1.1.5 of \cite{Sch99} for the notions of exactness for functors between quasi-Abelian categories.

\vspace{0.1in}
In this section, we recall the basic properties of the quasi-Abelian category of Banach modules and recall the foundation of the derived analytic geometry as discussed in \cite{BB16}, \cite{BK17}, \cite{BBK19}, and \cite{BK20}.

\subsection{Banach algebras}
\label{Banach algebras}

A {\it normed set} is a set $X$ equipped with a map $\v{\cdot}_X \colon X \to [0,\infty)$. We note that many authors assume the condition that $\set{x \in X}{\v{x}_X = 0}$ is a singleton or the condition that $\set{x \in X}{\v{x}_X = 0}$ is the empty set in the definition of a  normed set, but we assume neither of them. For normed sets $X_0$ and $X_1$, we denote by $X_0 \odot X_1$ the normed set given as the set $X_0 \times X_1$ equipped with the map
\begin{eqnarray*}
\v{\cdot}_{X_0 \odot X_1} \colon X_0 \times X_1 & \to & [0,\infty) \\
(x_0,x_1) & \mapsto & \v{x_0}_{X_0} \v{x_1}_{X_1}.
\end{eqnarray*}
For normed sets $X_0$ and $X_1$, a map $\phi \colon X_0 \to X_1$ is said to be {\it bounded} if there exists a $C \in [0,\infty)$ such that for any $x \in X_0$, the inequality $\v{\phi(x)}_{X_1} \leq C \v{x}_{X_0}$ holds. We call the minimum of such a $C$ {\it the operator norm of $\phi$}. We denote by $\NSet$ the category of normed sets and bounded maps, which naturally forms a symmetric monoidal category with respect to $\odot$.

\vspace{0.1in}
A {\it complete norm} on an Abelian group $M$ is a map $\n{\cdot} \colon M \to [0,\infty)$ such that the map $M^2 \to [0,\infty), \ (m_0,m_1) \mapsto \n{m_0 - m_1}$ is a complete metric on $M$. A {\it Banach Abelian group} is an Abelian group $M$ equipped with a complete norm $\n{\cdot}_M$ on it. We can always regard a Banach Abelian group as a normed set and a complete metric space. A Banach Abelian group is said to be {\it non-Archimedean} if it is an ultrametric space.

\vspace{0.1in}
We denote by $\BAb$ the category of Banach Abelian groups and bounded group homomorphisms and by $\NBAb \subset \BAb$ the full subcategory of non-Archimedean Banach Abelian groups. We denote by $\BAbleq \subset \BAb$ (resp.\ $\NBAbleq \subset \NBAb$) the subcategory with the same class of objects and whose class of morphisms is the subclass of morphisms of operator norm $\leq 1$. The hom functors on $\BAb$ and $\NBAb$ can be canonically enriched to internal hom functors with respect to the pointwise operations and the operator norm.
 
\vspace{0.1in}
We equip $\BAb$ and $\BAbleq$ (resp.\ $\NBAb$ and $\NBAbleq$) with the symmetric monoidal structures given by the completed tensor product assigning to each pair $(M_0,M_1)$ of objects the completion $M_0 \wh{\otimes} M_1$ of $M_0 \otimes_{\Z} M_1$ with respect to the uniformity associated to the tensor seminorm $M_0 \otimes_{\Z} M_1 \to [0,\infty)$ assigning to each $m \in M_0 \otimes_{\Z} M_1$ the infimum of $\sum_{j=0}^{n} \n{(m_{0,j},m_{1,j})}_{M_0 \odot M_1}$ (resp.\ $\max_{j=0}^{n} \n{(m_{0,j},m_{1,j})}_{M_0 \odot M_1}$) for an $(m_{0,j},m_{1,j})_{j=0}^{n} \in (M_0 \odot M_1)^{n+1}$ with $n \in \N$ and $\sum_{j=0}^{n} m_{0,j} \otimes m_{1,j} = m$.  The symmetric monoidal structures on $\BAb$ and $\NBAb$ are closed, as the internal hom functors give their right adjoint functors.

\vspace{0.1in}
A {\it complete norm on a ring $R$} is a complete norm $\n{\cdot}$ on its underlying Abelian group satisfying the property that there exists a $C \in [0,\infty)$ such that for any $(a_0,a_1) \in R^2$, the inequality $\n{a_0 a_1} \leq C \n{a_0} \ \n{a_1}$ holds. We say that $\n{\cdot}$ is {\it submultiplicative} if such a $C$ can be taken as $1$ and the equality $\n{1} = 1$ holds unless $R = \ens{0}$. For example, the Euclidean norm $\v{\cdot}_{\infty}$ is a submultiplicative complete norm on $\Z$, $\R$, and $\C$. A {\it Banach ring} is a ring $R$ equipped with a complete norm $\n{\cdot}_R$ on it.

\vspace{0.1in}
We note that many authors assume the submultiplicativity of the norm in the definition of a Banach ring but it is easy to check that every Banach ring is isomorphic to a Banach ring whose norm is submultiplicative. Indeed, if $(R, \v{\cdot})$ is a Banach ring, we can define the new norm on $R$ assigning $\sup \set{\frac{\v{x y}}{\v{y}}}{y \in R \setminus \ens{0}}$ to each $x \in R$ that is equivalent to the given one and is submultiplicative. We refer to \cite{BGR} Proposition 1.2.1/2 for a proof of this fact for $\NBAb$, but the proof works for $\BAb$ as well. The advantage of our definition is that it makes manifest that the notion of a monoid object of $\BAb$ (resp.\ $\NBAb$) for the tensor product $\wh{\otimes}$ is equivalent to that of a Banach ring (resp.\ non-Archimedean Banach ring) in our sense, while the submultiplicativity is a natural condition for the notion of monoid object of $\BAbleq$ (resp.\ $\NBAbleq$). In addition,  we recall that we assume the commutativity of rings because our aim is to study algebras of continuous functions, although we do not have to assume it in most parts of this paper. 

\vspace{0.1in}
Let $\cC$ denote $\BAb$ (resp.\ $\BAbleq$, $\NBAb$, $\NBAbleq$). Let $R$ be a monoid object of $\cC$. A {\it Banach $R$-module} is an $R$-module $M$ equipped with a complete norm $\n{\cdot}_M$ on its underlying Abelian group satisfying the property that there exists a $C \in [0,\infty)$ such that for any $(a,m) \in R \times M$, the inequality $\n{am}_M \leq C \n{a}_R \n{m}_M$ holds. The norm of a Banach $R$-module is said to be {\it submultiplicative} if such a $C$ can be taken as $1$. Then, the notion of $R$-module object of $\cC$ for the monoid object $R$ of $\cC$ is equivalent to that of a Banach $R$-module (resp.\ Banach $R$-module with submultiplicative norm, non-Archimedean Banach $R$-module, non-Archimedean Banach $R$-module with submultiplicative norm).

\vspace{0.1in}
We denote by $\Mod_{\cC}(R)$ the category of $R$-module objects of $\cC$ and $R$-linear homomorphisms of $\cC$. By \cite{BB16} Proposition 3.15 and \cite{BB16} Proposition 3.18, $\Mod_{\cC}(R)$ forms a quasi-Abelian category. The hom functor $\Hom_{\Mod_{\cC}(R)}$ can be canonically enriched to an internal hom functor with respect to the pointwise operations and the operator norm. We denote the internal hom functors by $\uHom_{\Mod_{\cC}(R)}$ to distinguish it from the usual hom-set functors denoted $\Hom_{\Mod_{\cC}(R)}$. We equip $\Mod_{\cC}(R)$ with the closed symmetric monoidal structure $\wh{\otimes}_R$ assigning to each pair $(M_0,M_1) \in \Mod_{\cC}(R)^2$ the cokernel  in $\cC$ of the difference of the left and right scalar multiplications $M_0 \wh{\otimes} R \wh{\otimes} M_1 \rightrightarrows M_0 \wh{\otimes} M_1$, which naturally forms an object of $\Mod_{\cC}(R)$.

\vspace{0.1in}
By construction, $\wh{\otimes}_R$ is naturally isomorphic to the completed tensor product obtained by replacing $\otimes_{\Z}$ in the definition of $\wh{\otimes}$ by $\otimes_R$, and satisfies the universality on bounded $R$-bilinear homomorphisms, i.e.\ $\Hom_{\Mod_{\cC}(R)}(M_0 \wh{\otimes} M_1,M_2)$ is in natural bijection with the set of bounded $R$-bilinear homomorphisms $M_0 \odot M_1 \to M_2$ for any $(M_0, M_1, M_2) \in \Mod_{\cC}(R)^3$. If $\cC$ is $\BAb$ or $\NBAb$, then the symmetric monoidal structure on $\Mod_{\cC}(R)$ is closed, as the internal hom functor gives its right adjoint functor.
 
\vspace{0.1in}
Let $M \in \Mod_{\cC}(R)$. We say that $M$ is {\it projective} if the external hom functor
\begin{eqnarray*}
\Hom_{\Mod_{\cC}(R)}(M,\cdot) \colon \Mod_{\cC}(R) \to \Ab
\end{eqnarray*}
is exact, or equivalently, by the left strong exactness of $\Hom_{\Mod_{\cC}(R)}$ with respect to the second argument, sends any strict epimorphism to a surjective map (cf.\ \cite{Sch99} Definition 1.3.18 and \cite{BK20} Proposition 2.4 (i)). We recall the simplest example of projective object for the reader's convenience. For an $r \in [0,\infty)$, we denote by $R_r$ the regular $R$-module object whose norm is rescaled by $r$ when $r > 0$ and the closed ideal $\ens{0} \subset R$ when $r = 0$. In particular, $R_1$ is the regular $R$-module object. We note that $R_r \cong R_1$ in $\Mod_{\cC}(R)$ for any $r \neq 0$, when $\cC$ is $\BAb$ or $\NBAb$ but this is not necessarily true if $\cC$ is $\BAbleq$ or $\NBAbleq$. Later on these modules will be used in universal constructions in the categories $\BAbleq$ and $\NBAbleq$. 

\begin{prp}
\label{projectivity of R}
If $\cC$ is $\BAb$ or $\NBAb$, $R_r$ is a projective object of $\Mod_{\cC}(R)$.
\end{prp}

\begin{proof}
If $r  = 0$, then $R_r$ is a zero object of $\Mod_{\cC}(R)$ and hence is a projective object of $\Mod_{\cC}(R)$. Suppose $r > 0$. 
Let $\phi \colon M_0 \to M_1$ be a strict epimorphism in $\Mod_{\cC}(R)$, and $\psi \colon R_r \to M_1$ a morphism in $\Mod_{\cC}(R)$. By the strictness of $\phi$, the morphism $\ol{\phi} \colon M_0/\ker(\phi) \to M_1$ associated to $\phi$ by the universality of the coimage is an isomorphism in $\Mod_{\cC}(R)$. Take a representative $m \in M_0$ of $\ol{\phi}^{-1}(\psi(1))$. Then, the map $\tl{\psi} \colon R_r \to M_0, \ a \mapsto a m$ is of operator norm $\leq \n{m}_M < \infty$, and hence is a morphism in $\Mod_{\cC}(R)$. By the definition, we have $\phi \circ \tl{\psi} = \psi$.
\end{proof}

We say that $M$ is {\it strongly flat} if the functor
\begin{eqnarray*}
(\cdot) \wh{\otimes}_R M \colon \Mod_{\cC}(R) \to \Mod_{\cC}(R)
\end{eqnarray*}
is strongly exact, or equivalently, by the right strong exactness of $\wh{\otimes}_R$, preserves the kernel of any (not necessarily strict) morphism (cf.\ \cite{BK20} Definition 3.2), and is {\it strongly internally injective} if the internal hom functor
\begin{eqnarray*}
\uHom_{\Mod_{\cC}(R)}(\cdot,M) \colon \Mod_{\cC}(R)^{\op} \to \Mod_{\cC}(R)
\end{eqnarray*}
is strongly exact, or equivalently, by the left strong exactness of $\uHom_{\Mod_{\cC}(R)}$ with respect to the first argument, preserves the cokernel of any (not necessarily strict) morphism, and we say that $M$ is {\it faithfully strongly internally injective} if the internal hom functor
\begin{eqnarray*}
\uHom_{\Mod_{\cC}(R)}(\cdot,M) \colon \Mod_{\cC}(R)^{\op} \to \Mod_{\cC}(R)
\end{eqnarray*}
preserves and reflects the cokernel of any (not necessarily strict) morphism.

\vspace{0.1in}
A {\it complete norm on an $R$-algebra} is a complete norm $\n{\cdot}$ on the underlying $R$-module and the underlying ring. A {\it Banach $R$-algebra} is an $R$-algebra $A$ equipped with a complete norm $\n{\cdot}_A$ on it. Then, the notion of a monoid object of $\Mod_{\cC}(R)$ is equivalent to that of a Banach $R$-algebra (resp.\ Banach $R$-algebra with submultiplicative norm, non-Archimedean Banach $R$-algebra, non-Archimedean Banach $R$-algebra with submultiplicative norm).

\begin{exm}
For a compact Hausdorff topological space $X$, we denote by $\rC(X,R)$ the set of continuous maps $X \to R$ equipped with the pointwise operations and the complete norm $\n{\cdot}_{\rC(K,R)}$ on it given as the supremum norm. The Banach $R$-algebra $\rC(X,R)$ is a natural extension of the notion of a $C^*$-algebra, as we will explain in the second paragraph of \S \ref{Derived analytic geometry}.
\end{exm}

We introduce an analytic analogue of the duality between the flatness and the injectivity. We abbreviate the internal hom functor
\begin{eqnarray*}
\uHom_{\Mod_{\cC}(R)}(\cdot, R_1) \colon \Mod_{\cC}(R)^{\op} \to \Mod_{\cC}(R)
\end{eqnarray*}
to $(\cdot)^{\vee}$. Since $R_1$ is the regular $R$-module object, $(\cdot)^{\vee}$ plays a role of the dual. Let $A$ be a monoid object of $\Mod_{\cC}(R)$. Then $A$ forms a monoid object of $\cC$ via the forgetful functor $\Mod_{\cC}(R) \to \cC$, and every object of $\Mod_{\cC}(A)$ is an object of $\Mod_{\cC}(R)$ via the restriction of scalars by the structure morphism, i.e.\ the unit $R \to A$. We denote by $\Res_A^R$ the forgetful functor $\Mod_{\cC}(A) \to \Mod_{\cC}(R)$.

\vspace{0.1in}
For any $(M_0,M_1) \in \Mod_{\cC}(A) \times \Mod_{\cC}(R)$, we denote by $\uHom_{\Mod_{\cC}(R)}^A(M_0,M_1) \in \Mod_{\cC}(A)$ the object $\uHom_{\Mod_{\cC}(R)}(\Res_A^R(M_0),M_1)$ of $\Mod_{\cC}(R)$ equipped with a natural action of $A$. By construction, $\uHom_{\Mod_{\cC}(R)}^A$ gives a functor $\Mod_{\cC}(A)^{\op} \times \Mod_{\cC}(R) \to \Mod_{\cC}(A)$, which is right adjoint to $\Res_A^R \circ \wh{\otimes}_A \colon \Mod_{\cC}(A)^2 \to \Mod_{\cC}(R)$ if $\cC$ is $\BAb$ or $\NBAb$, in the sense that for any $(M_0, M_1) \in \Mod_{\cC}(A)^2$ and any $M_2 \in \Mod_{\cC}(R)$, there is a natural isomorphism
\begin{eqnarray*}
\uHom_{\Mod_{\cC}(R)}(\Res_A^R(M_0 \wh{\otimes}_A M_1),M_2) \cong \Res_A^R(\uHom_{\Mod_{\cC}(A)}(M_0,\uHom_{\Mod_{\cC}(R)}^A(M_1,M_2)))
\end{eqnarray*}
that is even an isometry. We also abbreviate the functor
\begin{eqnarray*}
\uHom_{\Mod_{\cC}(R)}^A(\cdot, R_1) \colon \Mod_{\cC}(A)^{\op} \to \Mod_{\cC}(A).
\end{eqnarray*}
to $(\cdot)^{\vee}$.

\begin{prp}
\label{dual of flatness}
Suppose that $\cC$ is $\BAb$ or $\NBAb$, and $R_1$ is a faithfully strongly internally injective object of $\Mod_{\cC}(R)$. For any $M \in \Mod_{\cC}(A)$, $M$ is a strongly flat object of $\Mod_{\cC}(A)$ if and only if $M^{\vee}$ is a strongly internally injective object of $\Mod_{\cC}(A)$.
\end{prp}

\begin{proof}
We denote by $C$ a strictly exact sequence $0 \to M_0 \to M_1 \to M_2$ in $\Mod_{\cC}(A)$. The strict exactness of $C \wh{\otimes}_A M$ is equivalent to that of $(C \wh{\otimes}_A M)^{\vee}$ by the faithful strong internal injectivity of $R_1$, and hence is equivalent to that of $\uHom_{\Mod_{\cC}(A)}(C,M^{\vee})$ by the adjoint property of $\Res_A^R \circ \wh{\otimes}_A$ and $\uHom_{\Mod_{\cC}(R)}^A$. This implies the assertion.
\end{proof}

\subsection{Derived functors}
\label{Derived functors}

Henceforth, we only consider the case where $\cC$ is $\BAb$ or $\NBAb$. When $\cC = \BAb$, then we put $\cC_{\leq 1} \coloneqq \BAbleq$. When $\cC = \NBAb$, then we put $\cC_{\leq 1} \coloneqq \NBAbleq$. Let $R$ be a monoid object of $\cC$. We recall the construction of the functorial projective resolutions in $\Mod_{\cC}(R)$.

\vspace{0.1in}
For a family $(M_i)_{i \in I}$ of objects of $\Mod_{\cC}(R)$, the coproduct (resp.\ product) of the underlying objects of $\cC_{\leq 1}$, i.e.\ the completion of the algebraic direct sum with respect to the $\ell^1$-norm when $\cC = \BAb$ and the $\ell^{\infty}$-norm when $\cC = \NBAb$ (resp.\ the submodule of the algebraic direct product consisting of bounded families equipped with the $\ell^{\infty}$-norm), naturally forms an object of $\Mod_{\cC}(R)$, which we denote by $\wh{\bigoplus}^{\cC_{\leq 1}}_{i \in I} M_i$ (resp.\ $\prod^{\cC_{\leq 1}}_{i \in I} M_i$). We note that $\wh{\bigoplus}^{\cC_{\leq 1}}_{i \in I}$ (resp.\ $\prod^{\cC_{\leq 1}}_{i \in I}$) does not give a functor on $\Mod_{\cC}(R)$ unless $I$ is a finite set. Whereas, when $M_i$ is a zero object for all but finitely many $i \in I$, then the definition of $\wh{\bigoplus}^{\cC_{\leq 1}}_{i \in I} M_i$ (resp.\ $\prod^{\cC_{\leq 1}}_{i \in I} M_i$) extends the definition of the direct sum (resp.\ product) of the essentially finite family $(M_i)_{i \in I}$ in $\Mod_{\cC}(R)$ to a definable functor. In the case when $M_i$ is a zero object for all but finitely many $i \in I$, we will simplify the notation by denoting $\wh{\bigoplus}^{\cC_{\leq 1}}_{i \in I} M_i$ by $\bigoplus_{i \in I} M_i$ (resp.\ $\prod^{\cC_{\leq 1}}_{i \in I} M_i$ by $\prod_{i \in I} M_i$).

\vspace{0.1in}
We denote by $\Ch_{\cC}(R)$ the category of chain complexes of $\Mod_{\cC}(R)$ and chain homomorphisms. A morphism $f = (f_n)_{n \in \Z} \colon M = (M_n,d_{M_n})_{n \in \Z} \to N = (N_n,d_{N_n})_{n \in \Z}$ in $\Ch_{\cC}(R)$ is said to be a {\it quasi-isomorphism} if the mapping cone
\begin{eqnarray*}
\Cone(f) \coloneqq 
\left( M_{n+1} \oplus N_n, 
  \left(
    \begin{array}{cc}
       - d_{M_{n+1}} & 0 \\
       f_{n+1} & d_{N_n}
    \end{array}
   \right)
\right)_{n \in \Z}
\end{eqnarray*}
is strictly exact. We denote by $\cK_{\cC}(R)$ the {\it homotopy category} of $\Ch_{\cC}(R)$ that is obtained from $\Ch_{\cC}(R)$ by identifying homotopic chain homomorphisms.

\vspace{0.1in}
We define {\it the derived category $\Der_{\cC}(R)$ of $\Mod_{\cC}(R)$} as the right localisation of $\cK_{\cC}(R)$ by the right multiplicative system given by the class of homotopy classes of quasi-isomorphisms in $\Ch_{\cC}(R)$, or equivalently, the quotient of $\cK_{\cC}(R)$ by the null system of strictly exact complexes.

\vspace{0.1in}
We note that the notion of quasi-isomorphism for chain complexes of $\Mod_{\cC}(R)$ is not equivalent to that of a chain homomorphism which induces isomorphisms between the ``naively defined'' cohomology groups, i.e.\ the cohomology groups defined as
\begin{eqnarray*}
\rH_{\naive}^n(M) \coloneqq \frac{\ker(d_{M_n})}{\im(d_{M_{n-1}})}
\end{eqnarray*}
for a chain complex $M = (M_n,d_{M_n})_{n \in \Z}$, where $\im(d_{M_{n-1}})$ denotes the categorical image of $d_{M_{n -1}}$. For example, let $\v{\cdot}_0$ denote the trivial norm on $\Z$. The sequence
\begin{eqnarray*}
\cdots \to 0 \to (\Z,\v{\cdot}_{\infty}) \stackrel{\id_{\Z}}{\to} (\Z,\v{\cdot}_0) \to 0 \to \cdots
\end{eqnarray*}
is exact as a chain complex of $\Z$-modules, but is not strictly exact as a chain complex of $\Mod_{\cC}(\Z,\v{\cdot}_{\infty})$. On the other hand, the sequence
\begin{eqnarray*}
\cdots \to 0 \to \ell^1(\N,(\R,\v{\cdot}_{\infty})) \hookrightarrow \ell^2(\N,(\R,\v{\cdot}_{\infty})) \to 0 \to \cdots
\end{eqnarray*}
is neither exact as a chain complex of $\R$-vector spaces nor strictly exact in $\Ch_{\cC}(\R,\v{\cdot}_{\infty})$. These are typical examples of chain complexes in the analytic setting which are not strictly exact but whose naively defined cohomology groups vanish. It is possible to give cohomological criteria for checking strict exactness in $\Ch_{\cC}(R)$ but these involve the use of the left or right t-structure of $\Der_{\cC}(R)$ that will not be used in this work until \S \ref{Application to derived and non-derived descent} for the application to the descent theory.

\vspace{0.1in}
We recall the following construction. Let $M = (M_n, d_{M_n})_{n \in \Z} \in \Ch_{\cC}(R)$ and denote by $\Cyl(M) = (\Cyl(M)_n, d_{\Cyl(M)_n})_{n \in \Z}$ the {\it cylinder object} of $M$ defined as
\begin{eqnarray*}
(\Cyl(M)_n, d_{\Cyl(M)_n}) = 
\left( M_n \oplus M_{n+1} \oplus M_n,
  \left(
    \begin{array}{ccc}
    d_{M_n} & \id_{M_{n+1}} & 0 \\
    0 & - d_{M_{n+1}} & 0 \\
    0 & - \id_{M_{n+1}} & d_{M_n}
    \end{array}
  \right)
\right).
\end{eqnarray*}
The chain complex $\Cyl(M)$ is equipped with the chain homomorphisms $\pi \colon \Cyl(M) \to M$, defined as $\pi = (\id_{M_n},0,\id_{M_n})_{n \in \Z}$ and $i_0,i_1 \colon M \to \Cyl(M)$ defined as
\begin{eqnarray*}
i_0 = 
  \left(
    \begin{array}{c}
    \id_{M_n} \\
    0  \\
    0
    \end{array}
  \right)_{n \in \Z}, \ \ i_1 = 
  \left(
    \begin{array}{c}
    0 \\
    0  \\
    \id_{M_n}
    \end{array}
  \right)_{n \in \Z}.
\end{eqnarray*}
We note that $\pi \circ i_0 = \pi \circ i_1 = \id_M$ and one can easily check that $i_0 \circ \pi$, $i_1 \circ \pi$ and $\id_{\Cyl(M)}$ are homotopic chain homomorphisms. This shows that both $i_0$ and $i_1$ represent the inverse of the isomorphism in $\cK_{\cC}(R)$ represented by $\pi$. The most important property of $\Cyl(M)$ is that whenever two homotopic maps $f,g \colon M \to N$ are considered, then these factor through $\Cyl(M)$ as $f = r \circ i_0$ and $g = r \circ i_1$ for a morphism $r \colon \Cyl(M) \to N$.

\vspace{0.1in}
The following property follows from the work of Schneiders in \cite{Sch99}, but it is not obvious. We prove the proposition only for the quasi-Abelian categories $\Mod_{\cC}(R)$ but the same proof works for any quasi-Abelian category {\it mutatis mutandis}.

\begin{prp}
\label{qis saturated system}
The class of quasi-isomorphisms is closed under homotopy equivalences, and forms a saturated multiplicative system in $\Ch_{\cC}(R)$.
\end{prp}

\begin{proof}
The null system in $\cK_{\cC}(R)$ of strictly exact sequences coincides with the kernel of a cohomology functor by \cite{Sch99} Corollary 1.2.20, and hence the class of homotopy classes of quasi-isomorphisms in $\Ch_{\cC}(R)$ is a saturated multiplicative system in $\cK_{\cC}(R)$ by \cite{StackP} Lemma 13.6.11. This implies that the class of morphisms in $\Ch_{\cC}(R)$ homotopic to quasi-isomorphisms is a saturated multiplicative system in $\Ch_{\cC}(R)$. Therefore, it suffices to show that for any objects $M = (M_n,d_{M_n})_{n \in \Z}$ and $N = (N_n,d_{N_n})_{n \in \Z}$, any morphism $f \colon M \to N$ in $\Ch_{\cC}(R)$ homotopic to a quasi-isomorphism $g \colon M \to N$ in $\Ch_{\cC}(R)$ is a quasi-isomorphism.

\vspace{0.1in}
Consider the factorization $r \circ i_0 = f$ and $r \circ i_1 = g$ through the cylinder object $\Cyl(M)$. We obtain the distinguished triangles
\begin{eqnarray*}
\Cone(i_0) \to \Cone(f) & \to & \Cone(r) \to \Cone(i_0)[1] \\
\Cone(i_1) \to \Cone(g) & \to & \Cone(r) \to \Cone(i_1)[1]
\end{eqnarray*}
by the octahedral axiom, because $\cK_{\cC}(R)$ naturally forms a triangulated category (cf. \cite{Sch99} \S 1.2.1). Since $g$ is a quasi-isomorphism, $\Cone(g)$ is strictly exact. Therefore, in order to show that $f$ is a quasi-isomorphism, it suffices to show that $\Cone(i_0)$ and $\Cone(i_1)$ are strictly exact, by \cite{Sch99} Proposition 1.2.14. For this purpose, it suffices to show that a morphism in $\Ch_{\cC}(R)$ which represents an isomorphism in $\cK_{\cC}(R)$ is a quasi-isomorphism. By \cite{KS06} Exercise 11.5 the cone of any such morphism is isomorphic to $0$ in $\cK_{\cC}(R)$, and hence it is strictly exact by \cite{Sch99} Remark 1.2.2. This completes the proof.
\end{proof}

\begin{crl}
\label{derived isom is qis}
\begin{itemize}
\item[(i)] A morphism in $\Ch_{\cC}(R)$ represents an isomorphism in $\Der_{\cC}(R)$ if and only if it is a quasi-isomorphism in $\Ch_{\cC}(R)$.
\item[(ii)] A morphism in $\Der_{\cC}(R)$ is an isomorphism if and only if it is represented by a span of quasi-isomorphisms in $\Ch_{\cC}(R)$.
\item[(iii)] An $M \in \Ch_{\cC}(R)$ is strictly exact if and only if it is a zero object in $\Der_{\cC}(R)$.
\end{itemize}
\end{crl}

\begin{proof}
The assertions (i) and (ii) immediately follow from Proposition \ref{qis saturated system} and \cite{Sch20} Exercise 5.1. The assertion (iii) immediately follows from the assertion (i) applied to the zero morphism $M \to 0$ in $\Ch_{\cC}(R)$.
\end{proof}

For a normed set $X$, we put $\ell^{\cC}(X,R) \coloneqq \wh{\bigoplus}^{\cC_{\leq 1}}_{x \in X} R_{\v{x}_X}$. By construction, $\ell^{\cC}(\cdot,R)$ gives a functor $\NSet \to \Mod_{\cC}(R)$ that is left adjoint to the forgetful functor. Since $\wh{\bigoplus}^{\cC_{\leq 1}}$ preserves projectivity, $\ell^{\cC}(X,R)$ is a projective object of $\Mod_{\cC}(R)$ for any normed set $X$ by Proposition \ref{projectivity of R}. For any $M \in \Mod_{\cC}(R)$, the counit $\ell^{\cC}(M,R) \twoheadrightarrow M$ is a strict epimorphism by \cite{BK17} Lemma A.39 and \cite{BB16} Lemma 3.27, and permits to construct a projective resolution of $M$.

\vspace{0.1in}
Thus every object of $\Mod_{\cC}(R)$ admits a functorial projective resolution and hence $\Mod_{\cC}(R)$ has enough projective objects (cf.\ \cite{Sch99} Definition 1.3.20). In particular, for any additive functor $F$ from $\Mod_{\cC}(R)$ to a quasi-Abelian category, the class of projective objects of $\Mod_{\cC}(R)$ forms an $F$-projective additive subcategory (cf.\ \cite{Sch99} Definition 1.3.2) by \cite{Sch99} Remark 1.3.21, and hence $F$ is explicitly left derivable (cf.\ \cite{Sch99} Definition 1.3.1 and Definition 1.3.6).

\vspace{0.1in}
For an additive category $\varepsilon$, we denote by $\Chba(\varepsilon)$ the additive category of chain complexes bounded above of $\varepsilon$ and chain homomorphisms. We abbreviate the full subcategory $\Chba(\Mod_{\cC}(R)) \subset \Ch_{\cC}(R)$ to $\Chba_{\cC}(R)$. It admits a functorial projective resolution $P_{\cC}^{R} \colon \Chba_{\cC}(R) \to \Chba_{\cC}(R)$ by the argument above. For an $M = ((M_{i,j},\derv_{i,j})_{i \in \Z},(\derh_{i,j})_{i \in \Z})_{j \in \Z} \in \Chba(\Chba_{\cC}(R))$, we denote by $\Tot(M) \in \Chba_{\cC}(R)$ the total complex
\begin{eqnarray*}
\left( \bigoplus_{i \in \Z} M_{i,i+n}, \bigoplus_{i \in \Z} (\derv_{i,i+n} + (-1)^i \derh_{i,i+n}) \right)_{n \in \Z}
\end{eqnarray*}
of $M$ regarded as a double complex of $\Mod_{\cC}(R)$. We denote by $\Derba_{\cC}(R)$ the bounded above derived category of $\Mod_{\cC}(R)$, i.e.\ the localisation of the homotopy category of $\Chba_{\cC}(R)$ by the class of quasi-isomorphisms,  and by $\wh{\otimes}_R^{\bL} \colon \Derba_{\cC}(R)^2 \to \Derba_{\cC}(R)$ the left derived functor of $\hat{\otimes}_R$ assigning $\Tot(P_{\cC}^{R}(M) \wh{\otimes}_R P_{\cC}^{R}(N))$ to each $(M,N) \in \Derba_{\cC}(R)^2$. An argument completely parallel to that for the derived tensor product for algebraic modules implies the following:

\begin{prp}
\label{derived symmetric monoidal structure}
The following hold:
\begin{itemize}
\item[(i)] The tuple $(\Derba_{\cC}(R),\hat{\otimes}_R^{\bL},R)$ naturally forms a symmetric monoidal category.
\item[(ii)] For any $M \in \Mod_{\cC}(R)$ the restriction $\Mod_{\cC}(R) \to \Derba_{\cC}(R)$ of $M \wh{\otimes}_R^{\bL} (\cdot)$ is a left derived functor of $M \wh{\otimes}_R (\cdot)$.
\item[(iii)] For any $(M,N) \in \Chba_{\cC}(R)^2$ such that $N$ is termwise strongly flat, the natural morphism $\Tot(P_{\cC}^{R}(M) \wh{\otimes}_R P_{\cC}^{R}(N)) \to \Tot(M \wh{\otimes}_R N)$ in $\Chba_{\cC}(R)$ is a quasi-isomorphism.
\item[(iv)] For any morphism $f = (f_j)_{j \in \Z} \colon M \to N$ in $\Chba(\Chba_{\cC}(R))$ such that $f_j$ is a quasi-isomorphism for any $j \in \Z$, the natural morphism $\Tot(f) \colon \Tot(M) \to \Tot(N)$ in $\Chba_{\cC}(R)$ is a quasi-isomorphism.
\end{itemize}
\end{prp}

\begin{proof}
The assertion (i) follows from the symmetry and the associativity of the total complex. The assertion (ii) (resp.\ (iii)) follows from \cite{BK20} Proposition 3.11 and the assertion (iv) applied to the natural morphism
\begin{eqnarray*}
P_{\cC}^{R}(M) \wh{\otimes}_R P_{\cC}^{R}(N) \to M \wh{\otimes}_R P_{\cC}^{R}(N)
\end{eqnarray*}
in $\Chba_{\cC}(\Chba_{\cC}(R))$ for any $N \in \Mod_{\cC}(R)$ (resp.\ the natural morphisms
\begin{eqnarray*}
P_{\cC}^{R}(M) \wh{\otimes}_R P_{\cC}^{R}(N) \to P_{\cC}^{R}(M) \wh{\otimes}_R N \to M \wh{\otimes}_R N
\end{eqnarray*}
in $\Chba_{\cC}(\Chba_{\cC}(R))$). We show the assertion (iv). 

\vspace{0.1in}
Put $M = ((M_{i,j},\derh_{M_{i,j}})_{i \in \Z},(\derv_{M_{i,j}})_{i \in \Z})_{j \in \Z}$, $((N_{i,j},\derh_{N_{i,j}})_{i \in \Z},(\derv_{N_{i,j}})_{i \in \Z})_{j \in \Z}$, and $f_j = (f_{i,j})_{i \in \Z}$ for each $j \in \Z$. We have
\begin{eqnarray*}
& & \Cone(\Tot(f)) \\
& = & \left( \bigoplus_{i \in \Z} (M_{i+1,n-i} \oplus N_{i,n-i}), \bigoplus_{i \in \Z} 
  \left(
    \begin{array}{cc}
      - \derv_{M_{i+1,n-i}} - (-1)^{i+1} \derh_{M_{i+1,n-i}} & \derv_{N_{i,n-i}} + (-1)^i \derh_{N_{i,n-i}} \\
      f_{i+1,n-i} & 0 \\
    \end{array}
  \right)
\right)_{n \in \Z} \\
& = & \left( \bigoplus_{i \in \Z} \Cone(f_{n-i})_i, \bigoplus_{i \in \Z} (d_{\Cone(f_{n-i})_i} + (-1)^i (\derh_{M_{i+1,n-i}} + \derh_{N_{i,n-i}})) \right)_{n \in \Z},
\end{eqnarray*}
and hence the strict exactness of $\Cone(\Tot(f))$ at degree $n$ follows from the strict exactness of $\Cone(f_j)$ at degree $n - j$ applied inductively on $j \in \Z$.
\end{proof}

In particular, the derived tensor product can be computed by resolving only one argument by a flat resolution as usual in the algebraic setting.

\subsection{Derived analytic geometry}
\label{Derived analytic geometry}

We introduce the notion of homotopy epimorphism. When we deal with an appropriate formulation of the spectrum, the map of spectra induced by a homotopy epimorphism called a {\it formal homotopy Zariski open immersion}, because it plays the role of an open immersion in the derived geometry and especially in derived analytic geometry as formulated in \cite{BBK19}. Therefore a homotopy epimorphism is a derived counterpart of an open localisation. We usually drop the adjective ``formal'' when talking about formal homotopy Zariski open immersions as we will not study any ``non-formal'' ones.

\vspace{0.1in}
Let $R$ be a monoid object of $\cC$. We denote by $\Alg_{\cC}(R)$ the category of monoid objects of $\Mod_{\cC}(R)$ and monoid homomorphisms in $\Mod_{\cC}(R)$. Let $\pi \colon A_0 \to A_1$ be a morphism in $\Alg_{\cC}(R)$. We say that $\pi$ is a {\it homotopy epimorphism in $\Alg_{\cC}(R)$} if the multiplication $A_1 \wh{\otimes}_{A_0}^{\bL} A_1 \to A_1$ is an isomorphism in $\Derba_{\cC}(A_0)$. We note that $\wh{\otimes}_{A_0}$ gives a binary coproduct in $\Alg_{\cC}(A_0)$ by \cite{BK20} Proposition 2.16, and hence the non-derived multiplication $A_1 \wh{\otimes}_{A_0} A_1 \to A_1$ is an isomorphism in $\Mod_{\cC}(A_0)$ if and only if $\pi$ is an epimorphism in $\Alg_{\cC}(R)$. This fact clearly justifies the terminology ``homotopy epimorphism''.

\vspace{0.1in}
Since the simplest examples of epimorphisms in the category of $R$-algebras and $R$-algebra homomorphisms are given by quotients of localisations, which corresponds to pro-locally closed immersions of affine schemes over $\Spec(R)$, it is not difficult to imagine how a homotopy Zariski open immersion formally plays a role of an ``immersion''. We show the relation between the multiplications $A_1 \wh{\otimes}_{A_0} A_1 \to A_1$ and $A_1 \wh{\otimes}_{A_0}^{\bL} A_1 \to A_1$.

\begin{prp}
\label{homotopy epimorphism is epimorphism}
If $\pi$ is a homotopy epimorphism in $\Alg_{\cC}(R)$, then $\pi$ is an epimorphism in $\Alg_{\cC}(R)$.
\end{prp}

\begin{proof}
We put $P_{\cC}^{A_0}(A_1) = (P_n)_{n \in \Z}$, and by $\rho \colon P_{\cC}^{A_0}(A_1) \to A_1$ the canonical chain homomorphism of chain complexes of $\Mod_{\cC}(A_0)$. By Proposition \ref{derived symmetric monoidal structure} (ii), the natural morphism $\iota \colon \Tot(P_{\cC}^{A_0}(A_1) \wh{\otimes}_{A_0} P_{\cC}^{A_0}(A_1)) \to A_1 \wh{\otimes}_{A_0} P_{\cC}^{A_0}(A_1)$ in $\Chba_{\cC}(A_0)$ is a quasi-isomorphism. We denote by $\mu$ the multiplication $A_1 \wh{\otimes}_{A_0} A_1 \to A_1$, and by $\mu^{\bL}$ the composite $\mu \circ (\id_{A_1} \wh{\otimes}_{A_0} \rho) \circ \iota$. By definition, $\mu^{\bL}$ is a morphism in $\Chba_{\cC}(A_0)$ representing the multiplication $A_1 \wh{\otimes}_{A_0}^{\bL} A_1 \to A_1$. Since $\mu^{\bL}$ and $\iota$ represent isomorphisms in $\Derba_{\cC}(A_1)$, $\mu \circ (\id_{A_1} \wh{\otimes}_{A_0} \rho)$ represents an isomorphism in $\Derba_{\cC}(A_1)$. This implies that $\mu \circ (\id_{A_1} \wh{\otimes}_{A_0} \rho)$ is a quasi-isomorphism in $\Chba_{\cC}(R)$ by Corollary \ref{derived isom is qis} (i), and hence its mapping cone
\begin{eqnarray*}
\cdots \to A_1 \wh{\otimes}_{A_0} P_1 \to A_1 \wh{\otimes}_{A_0} P_0 \to A_1 \to 0 \to \cdots
\end{eqnarray*}
is strictly exact. On the other hand, the segment
\begin{eqnarray*}
A_1 \wh{\otimes}_{A_0} P_1 \to A_1 \wh{\otimes}_{A_0} P_0 \to A_1 \wh{\otimes}_{A_0} A_1 \to 0
\end{eqnarray*}
of the mapping cone of $\id_{A_1} \wh{\otimes}_{A_0} \rho$ is strictly coexact by the right strong exactness of $A_1 \wh{\otimes}_{A_0} (\cdot)$ and the strict coexactness of $P_1 \to P_0 \to A_1 \to 0$. In particular, $\mu$ is a morphism in $\Mod_{\cC}(A_0)$ under $A_1 \wh{\otimes}_{A_0} P_0$ between two cokernels of $A_1 \wh{\otimes}_{A_0} P_1 \to A_1 \wh{\otimes}_{A_0} P_0$, and hence is an isomorphism in $\Mod_{\cC}(A_0)$ by the universality of the cokernel. This implies that $\pi$ is an epimorphism in $\Alg_{\cC}(R)$.
\end{proof} 

Proposition \ref{homotopy epimorphism is epimorphism} implies that the notion of homotopy epimorphism is analogous to that of weak homological epimorphism for \Frechet $\C$-algebras introduced in \cite{AP20} Definition 3.14. The fact that a morphism between Stein spaces over $\C$ (resp.\ $C^{\infty}$-manifolds over $\R$) is an open immersion if and only if it induces a weak homological epimorphism between the associated \Frechet $\C$-algebras by \cite{AP20} Theorem 4.2 (resp.\ \cite{AP20} Theorem 5.3) is another evidence that a homotopy Zariski open immersion is a reasonable counterpart of an open immersion for the smooth setting.

\begin{prp}
\label{flat epimorphism is homotopy epimorphism}
Suppose that the natural morphism $A_1 \wh{\otimes}_{A_0}^{\bL} A_1 \to A_1 \wh{\otimes}_{A_0} A_1$ in $\Derba_{\cC}(A_0)$ is an isomorphism. Then $\pi$ is a homotopy epimorphism in $\Alg_{\cC}(R)$ if and only if $\pi$ is an epimorphism in $\Alg_{\cC}(R)$.
\end{prp}

\begin{proof}
The assertion follows from Proposition \ref{homotopy epimorphism is epimorphism}, because the morphism in $\Derba_{\cC}(A_0)$ in the assertion is represented by the morphism in $\Chba_{\cC}(A_0)$ given as the composite the natural morphism $\Tot(P_{\cC}^{A_0}(A_1) \wh{\otimes}_{A_0} P_{\cC}^{A_0}(A_1)) \to A_1 \wh{\otimes}_{A_0} A_1$ in $\Chba_{\cC}(R)$ and the multiplication $A_1 \wh{\otimes}_{A_0} A_1 \to A_1$ by the explicit formulation of $\wh{\otimes}^{\bL}$.
\end{proof}

It is quite remarkable that when $R$ is $\C$ equipped with $\v{\cdot}_{\infty}$, then the category $\CAlg$ of $C^*$-algebras and $*$-homomorphisms is a full subcategory of $\Alg_{\cC}(R)$. Indeed, every $C^*$-algebra is naturally identified with $\rC(X,\C)$ for some compact Hausdorff topological space $X$ by \GelfandNaimark theorem, and every $\C$-algebra homomorphism between $C^*$-algebra is a $*$-homomorphism because every character on a $C^*$-algebra is continuous and the involution on $\rC(X,\C)$ coincides with the pointwise complex conjugate on $X$. Therefore, our results about $\Alg_{\cC}(R)$ are applicable to $\CAlg$, and even to its full subcategory of von Neumann algebras. Although the category of $C^*$-algebras is not stable by the tensor product $\wh{\otimes}_R$, this will not prevent us to compute it in the category $\Alg_{\cC}(R)$ to deduce results for $C^*$-algebras by the following:

\begin{prp}
\label{epimorphism of C*-algebra}
Suppose that $R$ is $\C$ equipped with $\v{\cdot}_{\infty}$. Then a $\C$-algebra homomorphism $\pi$ between $C^*$-algebras is an epimorphism in $\CAlg$ if and only if $\pi$ is an epimorphism in $\Alg_{\cC}(R)$.
\end{prp}

\begin{proof}
By the argument above, $\pi$ is a morphism in $\CAlg$. Since $\CAlg$ is a full subcategory of $\Alg_{\cC}(R)$, $\pi$ is an epimorphism in $\CAlg$ if $\pi$ is an epimorphism in $\Alg_{\cC}(R)$. Suppose that $\pi$ is an epimorphism in $\CAlg$. By \GelfandNaimark theorem, the continuous map $\Phi$ between the associated spectra induced by $\pi$ is a monomorphism in the category $\CH$ of compact Hausdorff topological spaces and continuous maps. This implies that $\Phi$ is injective by the fact that the forgetful functor $\CH \to \Set$ is representable by a singleton, and hence $\pi$ is surjective by \GelfandNaimark theorem and Tietze extension theorem. Thus $\pi$ is an epimorphism in $\Alg_{\cC}(R)$.
\end{proof}

We use the notion of homotopy epimorphism to define the notion of cover for an object of $\Alg_{\cC}(R)$. In this context, objects of $\Alg_{\cC}(R)$ are identified with the objects of the dual category and hence formally considered as their own spectra. For a set $S$ and an $n \in \N$, we denote by $[S]^n \subset S^n$ the subset of $n$-tuples of distinct elements of $S$. Let $A \in \Alg_{\cC}(R)$. A {\it derived cover} of $A$ in $\cC$ is a subset $S \subset \Alg_{\cC}(A)$ satisfying the following:

\begin{itemize}
\item[(i)] For any $B \in S$, the structure morphism $A \to B$ is a homotopy epimorphism in $\Alg_{\cC}(R)$.
\item[(ii)] There exists a finite subset $S_0 \subset S$ satisfying derived Tate's acyclicity (cf.\ \cite{BK20} Theorem 2.15), i.e.\ the total complex  of the derived Tate--\Cech complex
\begin{eqnarray*}
0 \to A \to \prod_{B \in S_0} B \to \prod_{(B_0,B_1) \in [S_0]^2} B_0 \wh{\otimes}_A^{\bL} B_1 \to \cdots
\end{eqnarray*}
is strictly exact in $\Mod_{\cC}(R)$.
\end{itemize}
This notion of derived cover is very general and applies to any choice of $R$. For example, when $R = (\Z, \v{\cdot}_{\infty})$ then $\Alg_{\cC}(R)$ is the category of Banach rings and hence we obtain a definition of derived cover for any Banach ring. Moreover, the derived Tate--\Cech complex is a chain complex of $\Mod_{\cC}(A)$. Therefore its strict exactness depends only on $A$ and $S$, and does not depend on the structure map $R \to A$ nor on the (strongly exact) forgetful functor $\Mod_{\cC}(A) \to \Mod_{\cC}(R)$.

\vspace{0.1in}
We note that this notion of cover does not define a Grothendieck topology on $\Alg_{\cC}(R)$ because the objects $B_0 \wh{\otimes}_A^{\bL} B_1$ do not belong to the essential image of $\Alg_{\cC}(R)$ in general. This topology is well-defined only on the homotopy category of the opposite category of simplicial Banach $R$-algebras (cf.\ \cite{BBK19} Proposition 5.8). Nevertheless, it has been checked that this notion defines a topology when it is restricted to suitable subcategories of $\Alg_{\cC}(R)$, when $R$ is a non-Archimedean valued field (cf.\ \cite{BK17} Theorem 5.39). In this paper, we will prove new results of this kind for the algebras of continuous functions on compact Hausdorff spaces.

\vspace{0.1in}
Similarly, we introduce a terminology for the classical (non-derived) setting in order to contrast it with a derived cover. A {\it non-derived cover} of $A$ in $\cC$ is a subset $S \subset \Alg_{\cC}(A)$ satisfying the following:

\begin{itemize}
\item[(i)] For any $B \in S$, the structure morphism $A \to B$ is an epimorphism in $\Alg_{\cC}(R)$.
\item[(ii)] There exists a finite subset $S_0 \subset S$ satisfying Tate's acyclicity, i.e.\ the (non-derived) Tate--\Cech complex
\begin{eqnarray*}
0 \to A \to \prod_{B \in S_0} B \to \prod_{(B_0,B_1) \in [S_0]^2} B_0 \wh{\otimes}_A B_1 \to \cdots
\end{eqnarray*}
is strictly exact in $\Mod_{\cC}(R)$.
\end{itemize}
This notion of cover never defines a Grothendieck topology on $\Alg_{\cC}(R)^{\op}$ and there is no general theory associated to this notion of cover, in contrast with the derived covers that define a Grothendieck topology on the homotopy category of the opposite category of simplicial Banach algebras. We will show that for algebras of continuous functions on compact Hausdorff spaces, the notions of derived and non-derived covers are closely related.

\begin{prp}
\label{relation between derived covers and non-derived covers}
Let $S$ be a subset of $\Alg_{\cC}(A)$ such that the natural morphism
\begin{eqnarray*}
B_0 \wh{\otimes}_A^{\bL} B_1 \to B_0 \wh{\otimes}_A B_1
\end{eqnarray*}
in $\Derba_{\cC}(A)$ is an isomorphism for any $(B_0,B_1) \in S^2$. Then $S$ is a derived cover of $A$ in $\cC$ if and only if $S$ is a non-derived cover of $A$ in $\cC$.
\end{prp}

\begin{proof}
Let $S_0$ be a finite subset of $S$. The derived Tate--\Cech complex of $A$ associated to $S_0$ is isomorphic to the (non-derived) Tate--\Cech complex $C$ of $A$ associated to $S_0$ in $\Chba(\Derba_{\cC}(A))$ by the assumption. Therefore the assertion follows from Proposition \ref{derived symmetric monoidal structure} (iv) and Proposition \ref{flat epimorphism is homotopy epimorphism}.
\end{proof}

By Proposition \ref{relation between derived covers and non-derived covers}, Tate's acyclicity yields various examples of derived covers. We recall several known results on Tate's acyclicity for the reader's convenience.
\begin{itemize}
\item[(i)] \cite{Tat71} Theorem 8.2 for an affinoid $k$-algebra for a complete valuation field $k$ with a non-trivial valuation.
\item[(ii)] \cite{Ber90} Proposition 2.2.5 for a $k$-affinoid algebra for a complete valuation field $k$.
\item[(iii)] \cite{Hub94} Theorem 2.2 for a strongly Noetherian Tate ring.
\item[(iv)] \cite{Sch12} Theorem 6.3 (iii) for a perfectoid affinoid $k$-algebra for a perfectoid field $k$.
\item[(v)] \cite{Mih16} Theorem 3.9 for the Banach $k$-algebra of bounded continuous functions $X \to k$ for a topological space $X$ and a local field $k$.
\item[(vi)] \cite{Mih16} Theorem 4.8 for a stably uniform Banach $k$-algebra for a complete valuation field $k$.
\item[(vii)] \cite{BV18} Theorem 7 for a stably uniform Tate affinoid ring
\end{itemize}

Besides these positive results, it is well-known that Tate's acyclicity does not always hold for covers of the spectra of Banach rings. We refer to \cite{Hub94} pp.\ 520--521, \cite{Mih16} Corollary 3.2, \cite{Mih16} Theorem 4.6, and \cite{BV18} Proposition 18 for examples of simple covers by basic localisations that do not satisfy Tate's acyclicity. Even when the classical form of Tate's acyclicity fails derived Tate's acyclicity still holds, as explained in \cite{Sch19} Proposition 13.16 and \cite{BK20} Theorem 4.15.

\vspace{0.1in}
The aim of this paper is to study the relations between homotopy epimorphisms $\rC(X,R) \to \rC(Y,R)$ and closed immersions $Y \to X$, and relations between the notion of a derived cover of $\rC(X,R)$ and that of a topological cover of $X$ by closed subsets, for compact Hausdorff topological spaces $X$ and $Y$. We show relations of the former type in Proposition \ref{clopen subset is homotopy Zariski immersion}, Theorem \ref{closed immersion with projective kernel}, and Theorem \ref{closed immersion with flat kernel}, and relations of the latter type in Corollary \ref{clopen cover is derived cover}, Theorem \ref{relation of cover by subsets}, and Theorem \ref{relation of cover}.

\section{Homotopy Zariski open immersions versus closed immersions}
\label{Homotopy Zariski open immersions versus closed immersions}

Let $\cC$ be either $\BAb$ or $\NBAb$, $R$ a monoid object of $\cC$, $X$ a compact Hausdorff topological space, and $K$ a closed subset of $X$. We denote by $\pi_{X,R,K}$ the restriction map $\rC(X,R) \to \rC(K,R)$, and by $I_{X,R,K} \subset \rC(X,R)$ the kernel of $\pi_{X,R,K}$, i.e.\ the closed ideal $\set{f \in \rC(X,R)}{\forall x \in K, f(x) = 0}$. We study criteria for the property that $\pi_{X,R,K}$ is a homotopy epimorphism in $\Alg_{\cC}(R)$.

\subsection{Trivial examples}
\label{Trivial examples}

First, we give a trivial example of a homotopy epimorphism in $\Alg_{\cC}(R)$.

\begin{prp}
\label{clopen subset is homotopy Zariski immersion}
For any idempotent $e \in R$, $e R$ is a closed ideal of $R$, $R/e R$ is a projective object of $\Mod_{\cC}(R)$, and the quotient map $R \to R/e R$ is a homotopy epimorphism in $\Alg_{\cC}(R)$.
\end{prp}

In order to show Proposition \ref{clopen subset is homotopy Zariski immersion}, we prove the following lemma.

\begin{lmm}
\label{idempotent gives direct summand}
Let $M \in \Mod_{\cC}(R)$. For any idempotent $e \in R$, $e M$ and $(1-e)M$ are closed, the scalar multiplication $M \to M$ by $1-e$ induces an isomorphism $M/e M \to (1-e) M$ in $\Mod_{\cC}(R)$, and $M/e M$ is a direct summand of $M$.
\end{lmm}

\begin{proof}
Since the scalar multiplication $R \times M \to M$ is continuous, the kernel $e M$ (resp.\ $(1-e) M$) of the multiplication $M \to M$ by $1-e$ (resp.\ $e$) is closed. By the universality of the cokernel in $\Mod_{\cC}(R)$, the multiplication $M \to M$ by $1-e$ induces a bounded $R$-linear homomorphism $\mu \colon M/e M \to (1-e) M$. Its inverse is given by the composite of the inclusion $(1-e) M \hookrightarrow M$ and the canonical projection $M \twoheadrightarrow M/e M$, and hence is bounded. Therefore $\mu$ is an isomorphism in $\Mod_{\cC}(R)$. The addition $e M \oplus (1-e) M \to M$ is an isomorphism in $\Mod_{\cC}(R)$, because it admits the inverse given by the bounded $R$-linear homomorphism $M \to e M \oplus (1-e) M, \ m \mapsto (em,(1-e)m)$. Therefore $M/e M \cong (1-e) M$ is a direct summand of $M$.
\end{proof}

\begin{proof}[Proof of Proposition \ref{clopen subset is homotopy Zariski immersion}]
By Lemma \ref{idempotent gives direct summand}, $e R$ is closed and $R/e R$ is a direct summand of the regular $R$-module object $R_1$. This implies that $R/e R$ is a projective object of $\Mod_{\cC}(R)$ by Proposition \ref{projectivity of R}, and hence the natural morphism $(R/e R) \wh{\otimes}_R^{\bL} (R/e R) \to (R/e R) \wh{\otimes}_R (R/e R)$ in $\Derba_{\cC}(R)$ is an isomorphism, because it is represented by the natural morphism
\begin{eqnarray*}
P_{\cC}^{R}(R/e R) \wh{\otimes}_R P_{\cC}^{R}(R/e R) \to (R/e R) \wh{\otimes}_R (R/e R)
\end{eqnarray*}
in $\Chba_{\cC}(R)$, which is a quasi-isomorphism by Proposition \ref{derived symmetric monoidal structure} (iii) and \cite{BK20} Proposition 3.11. Since the canonical projection $R \twoheadrightarrow R/e R$ is surjective, the assertion follows from Proposition \ref{flat epimorphism is homotopy epimorphism}.
\end{proof}

Proposition \ref{clopen subset is homotopy Zariski immersion} informally means that every clopen subset of the spectrum of $R$ gives a homotopy Zariski open immersion with respect to any ``appropriate'' formulation of the spectrum, because idempotents precisely corresponds to a clopen subset in such a formulation. Indeed, Berkovich's spectrum satisfies this property by Shilov idempotent theorem (cf.\ \cite{Ber90} 7.4.1 Theorem). Similarly, a disjoint clopen cover of the spectrum of $R$ gives a derived and non-derived cover in the following sense:

\begin{crl}
\label{clopen cover is derived cover}
For any finite orthogonal system $E \subset R$ of idempotents with $\sum_{e \in E} e = 1$, the set $\set{R/e R}{e \in E}$ is a derived and non-derived cover of $R$ in $\cC$.
\end{crl}

\begin{proof}
The assertion immediately follows from Proposition \ref{relation between derived covers and non-derived covers} and Proposition \ref{clopen subset is homotopy Zariski immersion} by a repeated application of Lemma \ref{idempotent gives direct summand}.
\end{proof}

\subsection{Projectivity criterion}
\label{Projectivity criterion}

Secondly, we give a sufficient condition for $K \subset X$ to be such that $I_{X,R,K}$ is a projective object of $\Mod_{\cC}(\rC(X,R))$.  We also show how the projectivity of $I_{X,R,K}$ implies that $\pi_{X,R,K}$ is a homotopy epimorphism. A {\it bounded factorisation system of $R$} is a tuple $(V,C,D_0,D_1)$ of a neighbourhood $V \subset R$ of $0 \in R$, a constant $C \in [0,\infty)$, and continuous maps $D_0,D_1 \colon V \to R$ satisfying $a = D_0(a)D_1(a)$ and $\max \ens{\n{D_0(a)}_R,\n{D_1(a)}_R} \leq C \n{a}_R^{\frac{1}{2}}$ for any $a \in V$.

\begin{exm}
\label{bounded factrisation system}
\begin{itemize}
\item[(i)] Suppose that $0 \in R$ is isolated. Put $V \coloneqq \ens{0} \subset R$ and $C \coloneqq 1 \in [0,\infty)$. We define $D_0$ and $D_1$ as the inclusion $V \hookrightarrow R$. Then $(V,C,D_0,D_1)$ forms a bounded factorisation system. In particular, $\Z$ equipped with $\v{\cdot}_{\infty}$ and any ring equipped with the trivial norm have a bounded factorisation system.
\item[(ii)] Suppose that $R$ is $\R$ or $\C$ equipped with $\v{\cdot}_{\infty}$. Put $V \coloneqq R$ and $C \coloneqq 1 \in [0,\infty)$. We define $D_0$ as the map assigning $\v{x}_{\infty}^{\frac{1}{2}} \in [0,\infty) \subset R$ to each $x \in V$, and $D_1$ as the map assigning $\v{x}_{\infty}^{- \frac{1}{2}} x \in R$ to each $x \in V \setminus \ens{0}$ and $0 \in R$ to $0 \in V$. Then $(V,C,D_0,D_1)$ forms a bounded factorisation system of $R$.
\item[(iii)] Suppose that $R$ is a complete valuation field with non-trivial valuation. Take a $p \in R$ with $0 < \n{p}_R < 1$. Put $V \coloneqq R$ and $C \coloneqq \n{p}_R^{-1} \in [0,\infty)$. For each $x \in R \setminus \ens{0}$, we denote by $n_x \in \Z$ the greatest $n \in \Z$ satisfying $\n{x}_R \leq \n{p}_R^{2n}$. We define $D_0$ as the map assigning $p^{n_x} \in R$ to each $x \in V \setminus \ens{0}$ and $0 \in R$ to $0 \in V$, and $D_1$ as the map assigning $p^{-n_x} x \in R$ to each $x \in V \setminus \ens{0}$ and $0 \in R$ to $0 \in V$. Then $(V,C,D_0,D_1)$ forms a bounded factorisation system of $R$.
\item[(iv)] Suppose that $R$ has a bounded factorisation system $(V,C,D_0,D_1)$. Put $V_X \coloneqq \set{f \in \rC(X,R)}{\forall x \in X, f(x) \in V}$ and $C_X \coloneqq C$. For each $i \in \ens{0,1}$, we define $D_{X,i} \colon V_X \to \rC(X,R)$ as a map assigning $D_i \circ f$ to each $f \in V_X$. Then $(V_X,C_X,D_{X,0},D_{X,1})$ forms a bounded factorisation system of $\rC(X,R)$, and $I_{X,R,K}$ is stable under $D_{X,0}$ and $D_{X,1}$.
\end{itemize}
\end{exm}

We say that $K$ is a {\it $G_{\delta}$ set of $X$} if there is a non-empty countable set $\cU$ of open subsets of $X$ such that $\bigcap_{U \in \cU} U = K$.

\begin{thm}
\label{closed immersion with projective kernel}
If $X$ is totally disconnected, $R$ has a bounded factorisation system, and $K$ is a $G_{\delta}$ set of $X$, then $\pi_{X,R,K}$ is a homotopy epimorphism in $\Alg_{\cC}(R)$.
\end{thm}

We note that if $K$ is a clopen subset of $X$, then $K$ is a $G_{\delta}$ set of $X$ and $\pi_{X,R,K}$ is a homotopy epimorphism in $\Alg_{\cC}(R)$ by Proposition \ref{clopen subset is homotopy Zariski immersion}. Since a closed $G_{\delta}$ set of $X$ is not necessarily clopen, Theorem \ref{closed immersion with projective kernel} gives more non-trivial examples of homotopy epimorphisms. In order to verify Theorem \ref{closed immersion with projective kernel}, we prepare conventions and lemmata. We denote by $\CO(X)$ the set of clopen subsets of $X$.

\begin{lmm}
\label{ultranormality}
Let $F_0$ and $F_1$ be closed subsets of $X$ with $F_0 \cap F_1 = \emptyset$. If $X$ is totally disconnected, then there exists a $U \in \CO(X)$ such that $F_0 \subset U$ and $F_1 \subset X \setminus U$.
\end{lmm}

\begin{proof}
Since $X$ is a totally disconnected compact Hausdorff topological space, the topology of $X$ is generated by $\CO(X)$. For any $x \in F_0$, since $X \setminus F_1$ is an open neighbourhood of $x$, there exists a $U_x \in \CO(X)$ such that $x \in U_x \subset X \setminus F_1$. Put $\cU \coloneqq \set{U_x}{x \in F_0}$. Then $\cU$ is a clopen cover of $F_0$. By the compactness of $F_0$, there exists a finite subcover  $\cU_0 \subset \cU$ of $F_0$. Then $U \coloneqq \bigcup_{U \in \cU_0} U$ is a desired clopen subset of $X$.
\end{proof}

\begin{lmm}
\label{idempotent approximation}
If $X$ is totally disconnected and $K$ is a $G_{\delta}$ set of $X$, then $K$ is the intersection of a non-empty countable subset of $\CO(X)$.
\end{lmm}

\begin{proof}
Take a non-empty countable set $\cU$ of open subsets of $X$ with $\bigcap_{U \in \cU} U = K$. By Lemma \ref{ultranormality} applied to $(F_0,F_1) = (K,X \setminus U)$ for each $U \in \cU$, there exists a $(V_U)_{U \in \cU} \in \CO(X)^{\cU}$ such that $K \subset V_U \subset U$ for any $U \in \cU$. In particular, we have $K \subset \bigcap_{U \in \cU} V_U \subset \bigcap_{U \in \cU} U = K$, and hence $\bigcap_{U \in \cU} V_U = K$.
\end{proof}

\begin{lmm}
\label{finite intersection property}
Let $f \in I_{X,R,K}$ and $\epsilon \in (0,\infty)$. For any non-empty decreasing sequence $(F_n)_{n=0}^{\infty}$ of closed subsets of $X$ with $\bigcap_{n=0}^{\infty} F_n = K$, there exists an $n \in \N$ such that $\n{f(x)}_R < \epsilon$ for any $x \in F_n$.
\end{lmm}

\begin{proof}
The assertion immediately follows from the compactness of the closed subset $\set{x \in X}{\n{f(x)}_R \geq \epsilon}$.
\end{proof}

\begin{lmm}
\label{projectivity of I}
If $K$ is the intersection of a non-empty countable subset of $\CO(X)$, then $I_{X,R,K}$ is a projective object of $\Mod_{\cC}(\rC(X,R))$.
\end{lmm}

\begin{proof}
If $K = \emptyset$, then we have $I_{X,R,K} = \rC(X,R)$, and hence the projectivity of $I_{X,R,K}$ follows from Proposition \ref{projectivity of R}. Therefore we may assume $K \neq \emptyset$. Let $\phi \colon M_0 \to M_1$ be a strict epimorphism in $\Mod_{\cC}(\rC(X,R))$, and $\psi \colon I_{X,R,X} \to M_1$ a morphism in $\Mod_{\cC}(\rC(X,R))$. By the strictness of $\phi$, the morphism $\ol{\phi} \colon M_0/\ker(\phi) \to M_1$ associated to $\phi$ by the universality of the coimage is an isomorphism in $\Mod_{\cC}(\rC(X,R))$. We construct a morphism $\tl{\psi} \colon I_{X,R,K} \to M_0$ in $\Mod_{\cC}(\rC(X,R))$ such that $\phi \circ \tl{\psi} = \psi$. If $\psi = 0$, then $\tl{\psi}$ can be taken as $0$. Therefore we may assume $\psi \neq 0$. Take a non-empty countable subset $\cU \subset \CO(X)$ with $\bigcap_{U \in \cU} U = K$, and a surjective map $U_{\bullet} \colon \N \twoheadrightarrow \cU$. Replacing $\cU$ by $\ens{X} \cup \set{\bigcap_{i=0}^{n} U_i}{n \in \N}$, we may assume that $U_{\bullet}$ forms a decreasing sequence with $U_0 = X$. For each $i \in \N$, we denote by $\chi_i \in \rC(X,R)$ the characteristic function of $X \setminus U_i$, which belongs to $I_{X,R,K}$ by $K \subset U_i$. Let $i \in \N$. By $\psi \neq 0$, we have $R \neq \ens{0}$ and $M_1 \neq \ens{0}$. By $R \neq \ens{0}$ and $\emptyset \neq K \subset U_i$, we have $\n{\chi_i}_{\rC(X,R)} = \n{1}_R \neq 0$. By $M_1 \neq \ens{0}$, we have $C_0 \coloneqq \n{\ol{\phi}^{-1}}_{\sup} \in (0,\infty)$ and $C_1 \coloneqq \n{\psi}_{\sup} \in (0,\infty)$, where $\n{\cdot}_{\sup}$ denotes the operator norm. We obtain
\begin{eqnarray*}
\n{\ol{\phi}^{-1}(\psi(\chi_i))}_{M_0/\ker(\phi)} \leq C_0 C_1 \n{\chi_i}_{I_{X,R,K}} \leq C_0 C_1 \n{1}_R < 2 C_0 C_1 \n{1}_R,
\end{eqnarray*}
and hence there exists an $m_i \in M_0$ such that $m_i + \ker(\phi) = \ol{\phi}^{-1}(\psi(\chi_i))$ and $\n{m_i}_{M_0} < 2 C_0 C_1 \n{1}_R$. Let $f \in I_{X,R,K}$. We denote by $m(f)_{\bullet} \colon \N \to M_0$ the sequence assigning $f(m_0 - \chi_{n+1} m_n)$ to each $n \in \N$. We denote by $\mu_0 \colon R \wh{\otimes} M_0 \to M_0$ the scalar multiplication. Put $C_2 \coloneqq \n{\mu_0}_{\sup} \in (0,\infty)$. We have
\begin{eqnarray*}
& & \lim_{i_0 \to \infty} \sup_{i_1 \geq i_0} \n{m(f)_{i_1} - m(f)_{i_0}}_{M_0} = \lim_{i_0 \to \infty} \sup_{i_1 \geq i_0} \n{f(\chi_{i_0+1} m_{i_0+1} - \chi_{i_1+1} m_{i_1+1})}_{M_0} \\
& \leq & \lim_{i_0 \to \infty} \sup_{i_1 \geq i_0} C_2 \left( \n{f \chi_{i_0+1}}_{\rC(X,R)} \n{m_{i_0+1}}_{M_0} + \n{f \chi_{i_1+1}}_{\rC(X,R)} \n{m_{i_1+1}}_{M_0} \right) \\
& \leq & 2 C_0 C_1 C_2 \n{1}_R \lim_{i_0 \to \infty} \sup_{i_1 \geq i_0} \left( \n{f \chi_{i_0+1}}_{\rC(X,R)} + \n{f \chi_{i_1+1}}_{\rC(X,R)} \right) \\
& \leq & 4 C_0 C_1 C_2 \n{1}_R \lim_{i_0 \to \infty} \sup_{x \in U_{i_0+1}} \n{f(x)}_R = 0
\end{eqnarray*}
by Lemma \ref{finite intersection property}, and hence $m(f)_{\bullet}$ is a Cauchy sequence in $M_0$. We denote by $m(f)$ the limit of $m(f)_{\bullet}$ in $M_0$. Then the map $\ol{\psi} \colon I_{X,R,K} \to M_0, \ f \mapsto m(f)$ satisfies the desired property.
\end{proof}

\begin{lmm}
\label{idempotence of I for projective setting}
If $R$ has a bounded factorisation system and $K$ is the intersection of a non-empty countable subset of $\CO(X)$, then $I_{X,R,K} = I_{X,R,K}^2$ holds.
\end{lmm}

\begin{proof}
Take a bounded factorisation system $(V,C,D_0,D_1)$ of $R$. Let $f \in I_{X,R,K}$. We construct an $(f_0,f_1) \in I_{X,R,K}^2$ such that $f = f_0 f_1$. By Lemma \ref{finite intersection property}, there exists a $U \in \CO(X)$ such that $K \subset U$ and $f(x) \in V$ for any $x \in U$. We denote by $\chi \in I_{X,R,K}$ the characteristic function of $X \setminus U$. Put $f_0 \coloneqq f \chi + D_0 \circ (f(1 - \chi))$ and $f_1 \coloneqq \chi + D_1 \circ (f(1 - \chi))$. Then the pair $(f_0,f_1)$ satisfies the desired property.
\end{proof}

\begin{lmm}
\label{Idempotence implies the vanishing}
If $I_{X,R,K} = I_{X,R,K}^2$, then $\rC(K,R) \wh{\otimes}_{\rC(X,R)} I_{X,R,K} = \ens{0}$.
\end{lmm}

\begin{proof}
Let $f \in \rC(X,R) \otimes_{\rC(X,R)} I_{X,R,K}$, the algebraic tensor product. Take a presentation $f = \sum_{i=0}^{n} g_i \otimes h_i$ with $n \in \N$. By $I_{X,R,K} = I_{X,R,K}^2$, for each $i \in \N$ with $i \leq n$, there exists an $(h_{i,j,0},h_{i,j,1})_{j=0}^{n_i} \in (I_{X,R,K} \times I_{X,R,K})^{n_i+1}$ with $n_i \in \N$ and $h_i = \sum_{j=0}^{n_i} h_{i,j,0} h_{i,j,1}$. Then we have
\begin{eqnarray*}
& & f = \sum_{i=0}^{n} g_i \otimes h_i = \sum_{i=0}^{n} \sum_{j=0}^{n_i} g_i \otimes h_{i,j,0} h_{i,j,1} = \sum_{i=0}^{n} \sum_{j=0}^{n_i} g_i \pi_{X,R,K}(h_{i,j,0}) \otimes h_{i,j,1} \\
& = & \sum_{i=0}^{n} \sum_{j=0}^{n_i} g_i \cdot 0 \otimes h_{i,j,1} = 0.
\end{eqnarray*}
This implies $\rC(X,R) \otimes_{\rC(X,R)} I_{X,R,K} = \ens{0}$, and hence $\rC(X,R) \wh{\otimes}_{\rC(X,R)} I_{X,R,K} = \ens{0}$.
\end{proof}

\begin{lmm}
\label{approximating extension}
Let $f \in \rC(K,R)$ and $\epsilon \in (0,\infty)$. If $X$ is totally disconnected, then there exists an $\tl{f} \in \rC(X,R)$ such that $\n{\pi_{X,R,K}(\tl{f}) - f}_{\rC(K,R)} < \epsilon$ and $\n{\tl{f}}_{\rC(X,R)} \leq \n{f}_{\rC(K,R)} + \epsilon$.
\end{lmm}

\begin{proof}
For each $a \in R$, put $B(a) \coloneqq \set{x \in K}{\n{f(x) - a}_R < 2^{-1} \epsilon}$. Put $\cB \coloneqq \set{B(a)}{a \in R}$. Since $X$ is a totally disconnected compact Hausdorff topological space, the topology of $X$ is generated by $\CO(X)$. In addition, since $\cB$ is an open cover of the compact subset $K$, there exists a finite disjoint clopen cover $\cU$ of $K$ in $X$ such that $\set{K \cap U}{U \in \cU}$ is a refinement of $\cB$. Replacing $\cU$ by $\set{U \in \cU}{K \cap U \neq \emptyset}$, we may assume $K \cap U \neq \emptyset$ for any $U \in \cU$. For each $U \in \cU$, fix an $a \in R$ with $K \cap U \subset B(a)$, and denote by $\tl{f}_U$ the continuous map $X \to R$ which assigns $a$ to each $x \in U$ and $0$ to each $x \in X \setminus U$. Then $\tl{f} \coloneqq \sum_{U \in \cU} \tl{f}_U$ satisfies the desired property.
\end{proof}

\begin{lmm}
\label{Tietze for projective setting}
If $X$ is totally disconnected (resp.\ $X$ is totally disconnected and $R$ is a monoid object of $\cC_{\leq 1}$), then $\pi_{X,R,K}$ is a strict epimorphism in $\Mod_{\cC}(R)$ (resp.\ $\Mod_{\cC_{\leq 1}}(R)$).
\end{lmm}

\begin{proof}
It suffices to show that for any $(f,\epsilon) \in \rC(K,R) \times (0,\infty)$, there exists an $\tl{f} \in \rC(X,R)$ such that $\pi_{X,R,K}(\tl{f}) = f$ and $\n{\tl{f}}_{\rC(X,R)} \leq \n{f}_{\rC(K,R)} + \epsilon$. If $f = 0$, then $\tl{f}$ can be taken as $0 \in \rC(X,R)$. Therefore we may assume $f \neq 0$. By induction on $n \in \N$, we construct an $\tl{f}_n \in \rC(X,R)$ satisfying the following:
\begin{itemize}
\item[(i)] The inequality $\n{\pi_{X,R,K}(\sum_{i=0}^{n} \tl{f}_i) - f}_{\rC(K,R)} < 3^{-(n+1)} \epsilon$ holds for any $n \in \N$.
\item[(ii)] The inequality $\n{\tl{f}_0}_{\rC(X,R)} < \n{f}_{\rC(K,R)} + 3^{-1} \epsilon$ holds.
\item[(iii)] The inequality $\n{\tl{f}_n}_{\rC(X,R)} < 3^{-(n+1)} 4 \epsilon$ holds for any $n \in \N \setminus \ens{0}$.
\end{itemize}
When $n = 0$, there exists an $\tl{f}_n \in \rC(X,R)$ such that $\n{\pi_{X,R,K}(\tl{f}_n) - f}_{\rC(K,R)} < 3^{-1} \epsilon$ and $\n{\tl{f}_n}_{\rC(X,R)} \leq \n{f}_{\rC(K,R)} + 3^{-1} \epsilon$ by Lemma \ref{approximating extension} applied to $f$ and $3^{-1} \epsilon$. When $n > 0$, there exists an $\tl{f}_n \in \rC(X,R)$ such that $\n{\pi_{X,R,K}(\sum_{i=0}^{n} \tl{f}_i) - f}_{\rC(K,R)} < 3^{-(n+1)} \epsilon$ and $\n{\tl{f}_n}_{\rC(X,R)} \leq 3^{-{n+1}} 4 \epsilon$ by Lemma \ref{approximating extension} applied to $\pi_{X,R,X}(\sum_{i=0}^{n-1} \tl{f}_i) - f$ and $3^{-(n+1)} \epsilon$. By the completeness of $\rC(X,R)$, the infinite sum $\sum_{i=0}^{\infty} \tl{f}_i$ converges to a unique $\tl{f} \in \rC(X,R)$, which satisfies the desired property because $\n{f}_{\rC(K,R)} + 3^{-1} \epsilon + \sum_{i=1}^{\infty} 3^{-(n+1)} 4 \epsilon = \n{f}_{\rC(K,R)} + \epsilon$.
\end{proof}

\begin{proof}[Proof of Theorem \ref{closed immersion with projective kernel}]
We denote by $C \in \Chba_{\cC}(\rC(X,R))$ the chain complex
\begin{eqnarray*}
\cdots \to 0 \to I_{X,R,K} \to \rC(X,R) \to 0 \to \cdots
\end{eqnarray*}
of $\Mod_{\cC}(\rC(X,R))$, and by $f$ the morphism $C \to \rC(K,R)$ in $\Chba_{\cC}(\rC(X,R))$ associated to $\pi_{X,R,K}$. By Lemma \ref{Tietze for projective setting}, $\pi_{X,K,R}$ is an epimorphism in $\Alg_{\cC}(\rC(X,R))$ and $f$ is a quasi-isomorphism. By Proposition \ref{projectivity of R}, Lemma \ref{idempotent approximation}, and Lemma \ref{projectivity of I}, $C$ is termwise projective, and hence is termwise strongly flat by \cite{BK20} Proposition 3.11. We consider the commutative diagram
\begin{eqnarray*}
\xymatrix{
\Tot(P_{\cC}^{\rC(X,R)}(\rC(K,R)) \wh{\otimes}_{\rC(X,R)} P_{\cC}^{\rC(X,R)}(C)) \ar[r] \ar[d]^{\Tot(P_{\cC}^{\rC(X,R)}(\rC(K,R)) \wh{\otimes}_{\rC(X,R)} f)} & \rC(K,R) \wh{\otimes}_{\rC(X,R)} C \ar[d]^{\rC(K,R) \wh{\otimes}_{\rC(X,R)} f} \\
\Tot(P_{\cC}^{\rC(X,R)}(\rC(K,R)) \wh{\otimes}_{\rC(X,R)} P_{\cC}^{\rC(X,R)}(\rC(K,R))) \ar[r] & \rC(K,R) \wh{\otimes}_{\rC(X,R)} \rC(K,R)
}
\end{eqnarray*}
in $\Chba_{\cC}(\rC(X,R))$. Since $f$ is a quasi-isomorphism and $C$ is termwise strongly flat, the top horizontal arrow and the left vertical arrow are quasi-isomorphisms by Proposition \ref{derived symmetric monoidal structure} (iii) and (iv). We have $\rC(K,R) \wh{\otimes}_{\rC(X,R)} I_{X,R,K} = \ens{0}$ by Lemma \ref{idempotence of I for projective setting} and Lemma \ref{Idempotence implies the vanishing}, and hence the right vertical arrow is an isomorphism in $\Chba_{\cC}(\rC(X,R))$. This implies that the bottom horizontal arrow represents an isomorphism in $\Derba_{\cC}(\rC(X,R))$. Therefore the assertion follows from Proposition \ref{flat epimorphism is homotopy epimorphism}.
\end{proof}

Theorem \ref{closed immersion with projective kernel} has the following consequences.

\begin{crl}
If $X$ is a totally disconnected and first countable and $R$ has a bounded factorisation system, then $\pi_{X,R,\ens{x}}$ is a homotopy epimorphism in $\Mod_{\cC}(R)$ for any $x \in X$.
\end{crl}

\begin{proof}
Since $X$ is first countable, $x$ admits a countable fundamental system $\cU$ of neighbourhoods. Since $X$ is Hausdorff, $\cap_{U \in \cU} U = \ens{x}$. This implies that $\ens{x}$ is a $G_{\delta}$ set of $X$. Therefore, the assertion immediately follows Theorem \ref{closed immersion with projective kernel}.
\end{proof}

\begin{crl}
If $X$ is totally disconnected, $R$ has a bounded factorisation system, and there exists a pair $(R',f)$ of a Banach ring $R'$ and a continuous map $f \colon X \to R'$ such that $f^{-1}(\ens{0}) = K$, then $\pi_{X,R,K}$ is a homotopy epimorphism in $\Mod_{\cC}(R)$.
\end{crl}

\begin{proof}
By $K = \cap_{n \in \N} \set{x \in X}{\n{f(x)}_{R'} < 2^{-n}}$, $K$ is a $G_{\delta}$ set of $X$. Therefore, the assertion immediately follows Theorem \ref{closed immersion with projective kernel}.
\end{proof}

\subsection{Flatness criterion}
\label{Flatness criterion}

Thirdly, we deal with a setting for which we will show that $I_{X,R,K}$ is a flat object of $\Mod_{\cC}(\rC(X,R))$. We abbreviate $(\cC,R)$ to $\cD_2$ and $(\cC,X,R)$ to $\cD_3$. We say that $\cD_2$ is an {\it admissible pair} if it satisfies the following:
\begin{itemize}
\item[(i)] If $\cC$ is $\BAb$, then $R$ is $\R$ or $\C$ equipped with $\v{\cdot}_{\infty}$.
\item[(ii)] If $\cC$ is $\NBAb$, then $R$ is a finite field equipped with the trivial valuation or a local field.
\end{itemize}
In particular, in this case, $\pi_{X,R,K} \colon \rC(X,R) \to \rC(K,R)$ is a generalisation of the morphism between $C^*$-algebras associated to a closed immersion. We say that $\cD_3$ is an {\it admissible triple} if it satisfies the following:
\begin{itemize}
\item[(i)] The pair $\cD_2$ is an admissible pair.
\item[(ii)] If $\cC$ is $\NBAb$, then $X$ is totally disconnected.
\end{itemize}
We give an analogue of Theorem \ref{closed immersion with projective kernel}.

\begin{thm}
\label{closed immersion with flat kernel}
If $\cD_3$ is an admissible triple, then $\pi_{X,R,K}$ is a homotopy epimorphism in $\Alg_{\cC}(R)$.
\end{thm}

\begin{rmk}
\label{rational subsets}
We note that Theorem \ref{closed immersion with flat kernel} implies that Gerritzen--Grauert theorem does not hold in our setting unlike the classical rigid geometry, i.e.\ a homotopy epimorphism does not necessarily correspond to a finite union of rational subsets. For example, consider the case $X = \beta \N \setminus \N$. Since $\pi_{\beta \N,R,X}$ is surjective, as we will show in Lemma \ref{flatness of C(K,R)}, the cardinality of $\rC(X,R)$ is bounded by (and in fact equal to) $\aleph$ by the separability of $\beta \N$. This implies that the cardinality of the set of rational subsets of $X$ is bounded by $\aleph$. On the other hand, the cardinality of $X$ is $2^{\aleph}$, and hence admits a closed subset of cardinality $1$ which is not a rational subset.
\end{rmk}

In order to prove Theorem \ref{closed immersion with flat kernel}, we prepare some lemmata.

\begin{lmm}
\label{faithful injectivity of R}
If $\cD_2$ is an admissible pair, then $R_1$ is a faithfully strongly internally injective object of $\Mod_{\cC}(R)$.
\end{lmm}

\begin{proof}
The strong internal injectivity follows from Hahn--Banach theorem (cf.\ \cite{Ing52} Theorem 3 or \cite{Sch02} Proposition 9.2 for the non-Archimedean setting). Therefore, it suffices to show that for any sequence $M_0 \stackrel{f}{\to} M_1 \stackrel{g}{\to} M_2$ in $\Mod_{\cC}(R)$ such that the pair $(M_0^{\vee},f^{\vee})$ satisfies the universality of the cokernel of $g^{\vee}$ in $\Mod_{\cC}(R)$, the pair $(M_0,f)$ satisfies the universality of the kernel of $g$ in $\Mod_{\cC}(R)$. By the right strong exactness of $(\cdot)^{\vee}$ and the well-known consequence of Hahn--Banach theorem that the canonical natural transformation $\id_{\Mod_{\cC}(R)} \Rightarrow (\cdot)^{\vee \vee}$ consists of isometries, $f$ is a strict monomorphism, and hence the set-theoretic image $\im_{\Ab}(f)$ of $f$ is closed. Therefore it suffices to show that $\ker(g) = \im_{\Ab}(f)$.

\vspace{0.1in}
First, let $m \in \im_{\Ab}(f)$. Assume $m \notin \ker(g)$. Then, by Hahn--Banach theorem, there exists a $\mu \in M_2^{\vee}$ such that $\mu(g(m)) \neq 0$. Take an $m_0 \in M_0$ with $f(m_0) = m$. By $f^{\vee} \circ g^{\vee} = 0$, we have $0 = (f^{\vee} \circ g^{\vee})(\mu)(m_0) = g^{\vee}(\mu)(f(m_0)) = g^{\vee}(\mu)(m) = \mu(g(m))$, which contradicts $\mu(g(m)) \neq 0$. Therefore we obtain $m \in \ker(g)$. This implies $\im_{\Ab}(f) \subset \ker(g)$. 

\vspace{0.1in}
Next, let $m \in \ker(g)$. Assume $m \notin \im_{\Ab}(f)$. Then, by the closedness of $\im_{\Ab}(f)$ and Hahn--Banach theorem, there exists a $\mu \in M_1^{\vee}$ such that $\mu(m) \neq 0$ and $f^{\vee}(\mu) = 0$. We denote by $i$ the inclusion $\ker(g) \hookrightarrow M_1$. By $g \circ i = 0$, we have $i^{\vee} \circ g^{\vee} = 0$. By the universality of $(M_0^{\vee},f^{\vee})$ as the cokernel of $g^{\vee}$, there exists a unique $p_f \in \Hom_{\Mod_{\cC}(R)}(M_0^{\vee},\ker(g)^{\vee})$ such that $i^{\vee} = p_f \circ f^{\vee}$. We have $\mu(m) = i^{\vee}(\mu)(m) = (p_f \circ f^{\vee})(\mu)(m) = p_f(f^{\vee}(\mu))(m) = p_f(0)(m) = 0$, which contradicts $\mu(m) \neq 0$. Therefore, we obtain $m \in \im_{\Ab}(f)$. This implies $\ker(g) = \im_{\Ab}(f)$.
\end{proof}

We note that the assumptions in Lemma \ref{faithful injectivity of R} can be weakened, because the proof just requires $R$ to be a spherically complete valuation field in the non-Archimedean setting. We show an analytic counterpart of the stability by direct summands of the notion of injectivity.

\begin{lmm}
\label{section preserves flatness}
Let $(I_0, I_1) \in \Mod_{\cC}(R)^2$ admitting an $i \in \Hom_{\Mod_{\cC}(R)}(I_0,I_1)$ and a $\pi \in \Hom_{\Mod_{\cC}(R)}(I_1,I_0)$ with $\pi \circ i = \id_{I_0}$. If $I_1$ is strongly internally injective, then so are $I_0$ and $\ker(\pi)$.
\end{lmm}

\begin{proof}
The claim for $\ker(\pi)$ follows from the claim for $I_0$ because the inclusion $\ker(\pi) \hookrightarrow I_1$ is a split monomorphism. If $\pi = 0$, then $I_0 = \ens{0}$ by $\pi \circ i = \id_{I_0}$, and hence $I_0$ is strongly internally injective. Therefore we may assume $\pi \neq 0$. Let $0 \to M_0 \stackrel{f}{\to} M_1 \stackrel{g}{\to} M_2$ be a strictly exact sequence in $\Mod_{\cC}(R)$. For each $j \in \ens{0,1}$, we abbreviate $\Hom_{\Mod_{\cC}(R)}(\cdot,I_j)$ to $G_j$. We consider the following commutative diagram in $\Mod_{\cC}(R)$:
\begin{eqnarray*}
\xymatrix{
G_0(M_2) \ar[r]^-{G_0(g)} \ar[d]^-{i \circ (\cdot)} & G_0(M_1) \ar[r]^-{G_0(f)} \ar[d]^-{i \circ (\cdot)} & G_0(M_0) \ar[r] \ar[d]^-{i \circ (\cdot)} & 0  \ar@{=}[d] \\
G_1(M_2) \ar[r]^-{G_1(g)} \ar[d]^-{\pi \circ (\cdot)} & G_1(M_1) \ar[r]^-{G_1(f)} \ar[d]^-{\pi \circ (\cdot)} & G_1(M_0) \ar[r] \ar[d]^-{\pi \circ (\cdot)} & 0  \ar@{=}[d] \\
G_0(M_2) \ar[r]^-{G_0(g)} & G_0(M_1) \ar[r]^-{G_0(f)} & G_0(M_0) \ar[r] & 0.
}
\end{eqnarray*}
The composites of vertical arrows are identities by $\pi \circ i = \id_{I_0}$. The second horizontal sequence is strictly coexact by the strong internal injectivity of $I_1$. We show the exactness of the first horizontal sequence at $G_0(M_1)$. For this purpose, it suffices to show that for any $(\mu,\epsilon) \in \ker(G_0(f)) \times (0,\infty)$, there exists a $\tl{\mu} \in G_0(M_2)$ such that $\n{\mu - \tl{\mu} \circ g}_{G_0(M_1)} < \epsilon$. Put $C_{\pi} \coloneqq \n{\pi}_{\Hom_{\Mod_{\cC}(R)}(I_1,I_0)} \in [0,\infty)$. By $\pi \neq 0$, we have $C_{\pi} \neq 0$. By the strict coexactness of the second horizontal sequence, there exists a $\tl{\nu} \in G_1(M_2)$ such that $\n{i \circ \mu - \tl{\nu} \circ g}_{G_1(M_1)} < C_{\pi}^{-1} \epsilon$. Put $\tl{\mu} \coloneqq \pi \circ \tl{\nu} \in G_0(M_2)$. We have
\begin{eqnarray*}
\n{\mu - \tl{\mu} \circ g}_{G_0(M_1)} = \n{\pi \circ i \circ \mu - \pi \circ \tl{\nu} \circ g}_{G_0(M_1)} \leq C_{\pi} \n{i \circ \mu - \tl{\nu} \circ g}_{G_0(M_1)} < \epsilon.
\end{eqnarray*}
Therefore the first horizontal sequence is exact at $G_0(M_1)$.

\vspace{0.1in}
We show that $G_0(f)$ is a strict epimorphism. By the strict coexactness of the second horizontal sequence, $G_1(f)$ is a strict epimorphism. Therefore we obtain a $C \in [0,\infty)$ satisfying the property that for any $\nu \in G_1(M_0)$, there exists a $\tl{\nu} \in G_1(M_1)$ such that $\tl{\nu} \circ f = \nu$ and $\n{\tl{\nu}}_{G_1(M_1)} \leq C \n{\nu}_{G_1(M_0)}$. Put $C_i \coloneqq \n{i}_{\Hom_{\Mod_{\cC}(R)}(I_0,I_1)} \in [0,\infty)$. In order to show that $G_0(f)$ is a strict epimorphism, it suffices to show that for any $\mu \in G_0(M_0)$, there exists a $\tl{\mu} \in G_0(M_1)$ such that $\tl{\mu} \circ f = \mu$ and $\n{\tl{\mu}}_{G_0(M_1)} \leq C_{\pi} C C_i \n{\mu}_{G_0(M_0)}$. By the definition of $C$, there exists a $\tl{\nu} \in G_1(M_1)$ such that $\tl{\nu} \circ f = i \circ \mu$ and $\n{\tl{\nu}}_{G_1(M_1)} < C \n{i \circ \mu}_{G_1(M_0)}$. Put $\tl{\mu} \coloneqq \pi \circ \tl{\nu}$. We have $\tl{\mu} \circ f = \pi \circ \tl{\nu} \circ f = \pi \circ i \circ \mu = \mu$ and
\begin{eqnarray*}
\n{\tl{\mu}}_{G_0(M_1)} = \n{\pi \circ \tl{\nu}}_{G_0(M_1)} \leq C_{\pi} \n{\tl{\nu}}_{G_1(M_1)} \leq C_{\pi} C \n{i \circ \mu}_{G_1(M_0)} \leq C_{\pi} C C_i \n{\mu}_{G_0(M_0)}.
\end{eqnarray*}
Therefore $G_0(f)$ is a strict epimorphism. Thus the first horizontal sequence is strictly coexact.
\end{proof}

Let $I$ be a closed ideal of an $A \in \Alg_{\cC}(R)$. A {\it bounded approximate unit of $I$} is a directed set $(U,\leq_U)$ satisfying the following:
\begin{itemize}
\item[(i)] There exists a $C \in [0,\infty)$ such that $U \subset \set{u \in A}{\n{u}_A \leq C}$.
\item[(ii)] For any $(f,\epsilon) \in I \times (0,\infty)$, there exists a $u_0 \in U$ such that $\n{f - u f}_A < \epsilon$ for any $u \in U$ with $u_0 \leq_U u$.
\end{itemize}
We show a variant of \cite{Nem19} Proposition 3.2 in a way parallel to its original proof.

\begin{lmm}
\label{flatness of ideal admitting an approximate unit}
\begin{itemize}
\item[(i)] The regular $A$-module object is a strongly flat object of $\Mod_{\cC}(A)$.
\item[(ii)] If $\cD_2$ is an admissible pair and $I$ admits a bounded approximate unit, then $I$ and $A/I$ are strongly flat objects of $\Mod_{\cC}(A)$.
\end{itemize}
\end{lmm}

\begin{proof}
The assertion (i) immediately follows from Proposition \ref{projectivity of R} and \cite{BK20} Proposition 3.11. We show the assertion (ii). By Proposition \ref{dual of flatness}, Lemma \ref{faithful injectivity of R}, and the assertion (i), $A^{\vee}$ is a strongly internally injective object of $\Mod_{\cC}(A)$ and it suffices to show that $I^{\vee}$ and $(A/I)^{\vee}$ are strongly internally injective objects of $\Mod_{\cC}(A)$. For this purpose, it suffices to construct a section $i \colon I^{\vee} \to A^{\vee}$ of the dual $\pi \colon A^{\vee} \to I^{\vee}$ of the inclusion $I \hookrightarrow A$ by Lemma \ref{section preserves flatness} applied to $(I_0,I_1) = (I^{\vee},A^{\vee})$, because the dual $(A/I)^{\vee} \to A^{\vee}$ of the canonical projection $A \twoheadrightarrow A/I$ satisfies the universality of the kernel of $\pi$.

\vspace{0.1in}
Take a bounded approximate unit $(U,\leq_U)$ of $I$. Put $C \coloneqq \sup_{u \in U} \n{u}_A \in [0,\infty)$. We denote by $F$ the set of subsets $S \subset U$ admitting a $u \in U$ with $\set{u' \in U}{u \leq_U u'} \subset S$. Since $F$ forms a filter, there exists an ultrafilter $\wh{F}$ of $U$ containing $F$. Let $\mu \in I^{\vee}$. For any $f \in A$, $(\mu(uf))_{u \in U}$ is a net in the closed disc of $R$ of radius $C \n{\mu}_{I^{\vee}} \n{f}_A$, and hence admits a unique $\wh{F}$-limit point $\tl{\mu}(a) \in R$ by the compactness and the Hausdorffness of closed unit balls of $R$. By definition, $\tl{\mu}$ is an $R$-linear homomorphism $A \to R$ whose operator norm is bounded by $C \n{\mu}_{I^{\vee}}$, and hence is an element of $A^{\vee}$. The map $i \colon I^{\vee} \to A^{\vee}$ assigning $\tl{\mu}$ to each $\mu \in I^{\vee}$ gives an $R$-linear homomorphism whose operator norm is bounded by $C$, and hence gives a morphism in $\Mod_{\cC}(R)$. For any $(\mu,f) \in I^{\vee} \times I$, we have
\begin{eqnarray*}
& & \n{\mu(f) - \tl{\mu}(f)}_R = \n{\mu(f) - \lim_{\wh{F}} \mu(uf)}_R = \n{\lim_{\wh{F}} (\mu(f) - \mu(uf))}_R \\
& = & \n{\lim_{\wh{F}} \mu(f - uf)}_R = \lim_{\wh{F}} \n{\mu(f - uf)}_R = 0
\end{eqnarray*}
because of $F \subset \wh{F}$ and the continuity of $\n{\cdot}_R$, and hence $(\pi \circ i)(\mu)(f) = i(\mu)(f) = \tl{\mu}(f) = \mu(f)$. This implies that $i$ is a section of $\pi$.
\end{proof}

We note that the proof of Lemma \ref{flatness of ideal admitting an approximate unit} (ii) explains the reason why Theorem \ref{closed immersion with flat kernel} requires in the non-Archimedean setting the condition that $R$ is a finite field equipped with the trivial valuation or a local field, which is strictly stronger than the condition that $R$ is spherically complete, because it has been used the fact that the closed unit ball of $R$ is compact.

\begin{lmm}
\label{bounded approximate unit of I}
If $\cD_2$ is an admissible pair, then $I_{X,R,K}$ admits a bounded approximate unit contained in the closed unit disc.
\end{lmm}

\begin{proof}
When $\cC = \BAb$, we denote by $U \subset I_{X,R,K}$ the subset of continuous maps $u \colon X \to [0,1]$ with $\pi_{X,R,K}(u) = 0$ directed by the partial order $\leq_U$ given by the pointwise comparison. When $\cC = \NBAb$, we denote by $U \subset I_{X,R,K}$ the subset of continuous maps $u \colon X \to \ens{0,1}$ with $\pi_{X,R,K}(u) = 0$ directed by the partial order $\leq_U$ given by the inclusion of the corresponding clopen subsets. Then $(U,\leq_U)$ forms a bounded approximate unit of $I_{X,R,K}$ contained in the closed unit disc.
\end{proof}

We give a flat analogue of Lemma \ref{projectivity of I}.

\begin{lmm}
\label{flatness of I}
If $\cD_2$ is an admissible pair, then $I_{X,R,K}$ and $\rC(X,K)/I_{X,R,K}$ are strongly flat objects of $\Mod_{\cC}(\rC(X,R))$.
\end{lmm}

\begin{proof}
The assertion immediately follows from Proposition \ref{dual of flatness}, Lemma \ref{faithful injectivity of R}, Lemma \ref{flatness of ideal admitting an approximate unit} (ii), and Lemma \ref{bounded approximate unit of I}.
\end{proof}

We give an analogue of Lemma \ref{idempotence of I for projective setting}.

\begin{lmm}
\label{idempotence of I for flat setting}
If $\cD_2$ is an admissible pair, then $I_{X,R,K} = I_{X,R,K}^2$ holds.
\end{lmm}

\begin{proof}
The assertion immediately follows from Example \ref{bounded factrisation system} (ii) and (iv).
\end{proof}

We give an analogue of Lemma \ref{Tietze for projective setting}, which is an extension of Tietze extension theorem for compact Hausdorff topological spaces.

\begin{lmm}
\label{Tietze for flat setting}
If $\cD_3$ is an admissible triple, then $\pi_{X,R,K}$ is a strict epimorphism in $\Mod_{\cC}(R)$ and $\Mod_{\cC_{\leq 1}}(R)$.
\end{lmm}

\begin{proof}
When $\cC = \BAb$, then the assertion follows from Tietze extension theorem. When $\cC = \NBAb$, then $X$ is totally disconnected, and hence the assertion follows from Lemma \ref{Tietze for projective setting}.
\end{proof}

We note that the proof of Lemma \ref{Tietze for flat setting} explains the reason why Theorem \ref{closed immersion with flat kernel} requires the condition that $\cD_3$ is an admissible triple, which is strictly stronger than the condition that $\cD_2$ is an admissible pair in the non-Archimedean setting. Indeed, the following counterexample shows that the condition is actually necessary:

\begin{rmk}
\label{closed immersion is not necessarily homotopy epimorphism}
Consider the case $(\cC,X,R,K) = (\NBAb,[0,1],\Qp,\ens{0,1})$, where $\cD_2$ is an admissible pair but $\cD_3$ is not an admissible triple. Then $\pi_{X,R,K}$ is not surjective because $\rC(X,R) \cong \Qp$ and $\rC(K,R) \cong \Qp^2$, and hence is not a strict epimorphism in $\Mod_{\cC}(R)$. Moreover, $\pi_{X,R,K}$ is not an epimorphism in $\Alg_{\cC}(R)$ by $\dim_{\Qp}(\Qp^2 \wh{\otimes}_{\Qp} \Qp^2) = \dim_{\Qp}(\Qp^4) = 4 \neq 2 = \dim_{\Qp}(\Qp^2)$, and hence is not a homotopy epimorphism in $\Alg_{\cC}(R)$ by Proposition \ref{homotopy epimorphism is epimorphism}. In particular, Theorem \ref{closed immersion with flat kernel} and Lemma \ref{Tietze for flat setting} would not hold if we dropped the condition that $\cD_3$ is an admissible triple.
\end{rmk}

\begin{lmm}
\label{flatness of C(K,R)}
If $\cD_3$ is an admissible triple, then $\rC(K,R)$ is a strongly flat object of $\Mod_{\cC}(\rC(X,R))$, and the morphism $\rC(X,R)/I_{X,R,K} \to \rC(K,R)$ in $\Alg_{\cC_{\leq 1}}(\rC(X,R))$ induced by $\pi_{X,R,K}$ is an isomorphism.
\end{lmm}

\begin{proof}
The assertion immediately follows from Lemma \ref{flatness of ideal admitting an approximate unit}, Lemma \ref{flatness of I}, and Lemma \ref{Tietze for flat setting}.
\end{proof}

\begin{proof}[Proof of Theorem \ref{closed immersion with flat kernel}]
By Lemma \ref{flatness of C(K,R)}, $C(K,R)$ is a strongly flat object of $\rC(X,R)$ and $\pi_{X,R,K}$ is an epimorphism in $\Alg_{\cC}(\rC(X,R))$. The natural morphism
\begin{eqnarray*}
\rC(K,R) \wh{\otimes}_{\rC(X,R)}^{\bL} \rC(K,R) \to \rC(K,R) \wh{\otimes}_{\rC(X,R)} \rC(K,R)
\end{eqnarray*}
in $\Derba_{\cC}(\rC(X,R))$ is an isomorphism because it is represented by the morphism
\begin{eqnarray*}
\Tot(P_{\cC}^{\rC(X,R)}(\rC(K,R)) \wh{\otimes}_{\rC(X,R)}^{\bL} P_{\cC}^{\rC(X,R)}(\rC(K,R))) \to \rC(K,R) \wh{\otimes}_{\rC(X,R)} \rC(K,R)
\end{eqnarray*}
in $\Chba_{\cC}(\rC(X,R))$, which is a quasi-isomorphism by the strong flatness of $C(K,R)$ and Proposition \ref{derived symmetric monoidal structure} (iii). Since $\pi_{X,R,K}$ is an epimorphism in $\Alg_{\cC}(\rC(X,R))$, the assertion follows from Proposition \ref{flat epimorphism is homotopy epimorphism}.
\end{proof}

\section{Derived covers versus topological covers}
\label{Derived covers versus topological covers}

Let $\cC$ be either $\BAb$ or $\NBAb$, $R$ a monoid object of $\cC$, and $X$ a compact Hausdorff topological space. We continue to use the abbreviations $\cD_2$ and $\cD_3$ as in \S \ref{Flatness criterion}. We have shown that a closed immersion of topological spaces corresponds to a homotopy Zariski localisations of the associated algebras of continuous functions, under mild assumptions on $\cD_3$. Then, natural questions arise: Does a homotopy Zariski open immersion always correspond to a closed immersion? Does a topological cover by closed immersions correspond to a cover in the sense of derived analytic geometry and vice versa? We will answer these two questions in this section in Theorem \ref{relation of cover by subsets}, Theorem \ref{relation of cover}, and Lemma \ref{equivalemce of homotopy epimorphism and strict epimorphism}.

\subsection{Topological covers}
\label{Topological covers}

We introduce the notion of a topological cover. We denote by $\TDCH \subset \CH$ (cf.\ the proof of Proposition \ref{epimorphism of C*-algebra}) the full subcategory of totally disconnected compact Hausdorff spaces. When $\cC = \BAb$, then we put $\CH_{\cC} \coloneqq \CH$. When $\cC = \NBAb$, then we put $\CH_{\cC} \coloneqq \TDCH$. We denote by $\CH/X$ the slice category of $\CH$ over $X$, and by $\TDCH/X$ the full subcategory of $\CH/X$ consisting of objects whose underlying topological spaces are totally disconnected. We note that $\TDCH/X$ makes sense even when $X$ is not totally disconnected.

\vspace{0.1in}
We denote by $C^*_{\cD_3}$ the full subcategory of $\Alg_{\cC}(\rC(X,R))$ consisting of objects the form $\rC(Y,R)$ for some $Y \in \CH_{\cC}$, and by $\Gamma^{\cD_3}$ the functor $(\CH_{\cC}/X)^{\op} \to C^*_{\cD_3}$ assigning $\rC(Y,R)$ equipped with the pre-composition $\rC(X,R) \to \rC(Y,R)$ with $\Phi$ to each $(Y,\Phi) \in \CH_{\cC}/X$. For a subset $S \subset \CH_{\cC}/X$, we put $\Gamma^{\cD_3}_*(S) \coloneqq \set{\Gamma^{\cD_3}(Y,\Phi)}{(Y,\Phi) \in S}$.

\vspace{0.1in}
When $\cC = \BAb$, we denote by $\beta_{\cC}$ the identity functor $\CH \to \CH$, and by $\iota_{\cC}$ the identity natural transformation $\id_{\CH} \Rightarrow \beta_{\cC}$. When $\cC = \NBAb$, we denote by $\beta_{\cC}$ the Banaschewski compactification functor $\CH \to \TDCH$, and by $\iota_{\cC}$ the natural transformation $\id_{\CH} \Rightarrow \beta_{\cC}$ assigning the canonical map $Y \to \beta_{\cC}(Y)$ to each compact Hausdorff topological space $Y$. We recall that the Banaschewski compactification functor is a left adjoint functor to the inclusion $\TDCH \hookrightarrow \CH$ with unit $\iota_{\cC}$, and is explicitly described in \cite{Mih14} Theorem 1.3 using the Berkovich spectrum functor $\Alg_{\cC}(R) \to \CH$.

\vspace{0.1in}
A {\it $\cC$-embedding} is a morphism $\phi \colon X_0 \to X_1$ in $\CH$ such that $\beta_{\cC}(\phi)$ is a homeomorphism onto the image. In particular, the notion of an $\BAb$-embedding is equivalent to that of a homeomorphism onto image between compact Hausdorff topological spaces, and every homeomorphism onto the image between totally disconnected compact Hausdorff topological spaces is an $\NBAb$-embedding. On the other hand, an $\NBAb$-embedding is not necessarily a homeomorphism onto the image, and a homeomorphism onto the image between (not necessarily totally disconnected) compact Hausdorff topological spaces is not necessarily an $\NBAb$-embedding. Indeed, the projection $[0,1] \to \ens{0}$ is an $\NBAb$-embedding which is not a homeomorphism onto the image, and the inclusion $\ens{0,1} \hookrightarrow [0,1]$ is a homeomorphism onto the image which is not an $\NBAb$-embedding.

\vspace{0.1in}
A {\it $\cC$-cover of $X$} is a subset $S \subset \CH/X$ satisfying the following:
\begin{itemize}
\item[(i)] For any $Y \in S$, the structure map $Y \to X$ is a $\cC$-embedding.
\item[(ii)] There exists a finite subset $S_0 \subset S$ such that the composite of $\iota_{\cC}(X) \colon X \to \beta_{\cC}(X)$ and the coproduct $\bigsqcup_{Y \in S_0} Y \to X$ of the structure maps is surjective.
\end{itemize}
We now describe the relations between $\cC$-covers of $X$, derived covers of $\rC(X,R)$ in $\cC$, and non-derived covers of $\rC(X,R)$ in $\cC$.

\subsection{Acyclicity for a topological cover by closed immersions}
\label{Acyclicity for a topological cover by closed immersions}

We first compare the notion of a $\cC$-cover by closed subsets and the notion of a derived cover.

\begin{thm}
\label{relation of cover by subsets}
Let $S$ be a set of closed immersions into $X$. If $\cD_3$ is an admissible triple, then the following are equivalent:
\begin{itemize}
\item[(i)] The set $S$ is a $\cC$-cover of $X$.
\item[(ii)] The set $S$ admits a finite subset $S_0 \subset S$ with $\bigcup_{(K,j) \in S_0} j(K) = X$.
\item[(iii)] The set $\Gamma^{\cD_3}_*(S)$ is a non-derived cover of $\rC(X,R)$ in $\cC$.
\item[(iv)] The set $\Gamma^{\cD_3}_*(S)$ is a derived cover of $\rC(X,R)$ in $\cC$.
\end{itemize}
\end{thm}

We note that the assumption that $\cD_3$ is an admissible triple implies that all closed subsets of $X$ are objects of $\CH_{\cC}$. Therefore, we can ignore $\beta_{\cC}$ and $\iota_{\cC}$ in the definitions of a $\cC$-embedding and a $\cC$-cover (these maps will be used in the next subsection). By an argument similar to that of Remark \ref{rational subsets}, a $\cC$-cover does not necessarily contain a finite subcover by finite unions of rational subsets. Therefore, Theorem \ref{relation of cover by subsets} is strictly stronger than Tate's acyclicity for a rational cover. In order to show Theorem \ref{relation of cover by subsets}, we prove some lemmata.

\begin{lmm}
\label{subset operation and aritmetic operation}
Let $K_0$ and $K_1$ be closed subsets of $X$. If $\cD_3$ is an admissible triple, then the closure of $I_{X,R,K_0} I_{X,R,K_1}$ coincides with $I_{X,R,K_0 \cup K_1}$ and $I_{X,R,K_0} \cap I_{X,R,K_1}$, and the closure of $I_{X,R,K_0} + I_{X,R,K_1}$ coincides with $I_{X,R,K_0 \cap K_1}$.
\end{lmm}

\begin{proof}
We have  $I_{X,R,K_0 \cup K_1} = I_{X,R,K_0} \cap I_{X,R,K_1}$, because an $f \in \rC(X,R)$ is identically zero on $K_0 \cup K_1$ if and only if $f$ is identically zero on $K_0$ and $K_1$. Since $I_{X,R,K_0 \cup K_1}$ and $I_{X,R,K_0 \cap K_1}$ are closed ideals containing $I_{X,R,K_0} I_{X,R,K_1}$ and $I_{X,R,K_0} + I_{X,R,K_1}$ respectively, it suffices to show that $I_{X,R,K_0} I_{X,R,K_1}$ and $I_{X,R,K_0} + I_{X,R,K_1}$ are dense in $I_{X,R,K_0 \cup K_1}$ and $I_{X,R,K_0 \cap K_1}$ respectively.

\vspace{0.1in}
We show that for any $(f,\epsilon) \in I_{X,R,K_0 \cup K_1} \times (0,\infty)$, there exists an $(f_0,f_1) \in I_{X,R,K_0} \times I_{X,R,K_1}$ such that $\n{f - f_0 f_1}_{\rC(X,R)} \leq \epsilon$. By Lemma \ref{bounded approximate unit of I}, $I_{X,R,K_0}$ admits a bounded approximate unit $(U,\leq_U)$ contained in the closed unit disc. By $f \in I_{X,R,K_0 \cup K_1} \subset I_{X,R,K_0}$, there exists a $u \in U$ such that $\n{f - u f}_{\rC(X,R)} < \epsilon$. Put $f_0 \coloneqq u$ and $f_1 \coloneqq f$. Then $(f_0,f_1)$ satisfies the desired property. Therefore $I_{X,R,K_0} I_{X,R,K_1}$ is dense in $I_{X,R,K_0 \cup K_1}$.

\vspace{0.1in}
We show that for any $(f,\epsilon) \in I_{X,R,K_0 \cap K_1} \times (0,\infty)$, there exists an $(f_0,f_1) \in I_{X,R,K_0} \times I_{X,R,K_1}$ such that $\n{f - (f_0 + f_1)}_{\rC(X,R)} \leq \epsilon$. Put $K_2 \coloneqq \set{x \in K_1}{\n{f(x)}_R \geq 3^{-2} \epsilon}$. By $K_0 \cap K_1 \subset f^{-1}(\ens{0})$, we have $K_0 \cap K_2 = \emptyset$. By the normality of $X$, the characteristic function $\chi_0 \colon K_0 \cup K_2 \to R$ of $K_2$ is continuous. By Lemma \ref{Tietze for flat setting} applied to $K = K_0 \cup K_2$, there exists a $\tl{\chi}_0 \in \rC(X,R)$ such that $\pi_{X,R,K_0 \cup K_2}(\tl{\chi}_0) = \chi_0$ and $\n{\tl{\chi}_0}_{\rC(X,R)} < 2$. By $\pi_{X,R,K_0}(\tl{\chi}_0) = \pi_{K_0 \cup K_2,R,K_0}(\chi_0) = 0$, we have $\tl{\chi}_0 \in I_{X,R,K_0}$. Put $f_0 \coloneqq f \tl{\chi}_0 \in I_{X,R,K_0}$ and $K_3 \coloneqq \set{x \in X}{\n{(f - f_0)(x)}_R \geq 3^{-1} \epsilon}$. For any $x \in K_2$, we have 
\begin{eqnarray*}
(f - f_0)(x) = f(x) - f(x) \tl{\chi}_0(x) = f(x) - f(x) \chi_0(x) = 0.
\end{eqnarray*}
For any $x \in K_1 \setminus K_2$, we have
\begin{eqnarray*}
& & \n{(f - f_0)(x)}_R = \n{f(x) - f(x) \tl{\chi}_0(x)}_R = \n{f(x)}_R \n{1 - \tl{\chi}_0(x)}_R \\
& \leq & \n{f(x)}_R (\n{1}_{\rC(X,R)} + \n{\tl{\chi}_0}_{\rC(X,R)}) < 3^{-2} \epsilon \cdot (1 + 2) = 3^{-1} \epsilon.
\end{eqnarray*}
This implies that $K_1 \cap K_3 = \emptyset$. By the normality of $X$, the characteristic function $\chi_1 \colon K_1 \cup K_3 \to R$ of $K_3$ is continuous. By Lemma \ref{Tietze for flat setting} applied to $K = K_1 \cup K_3$, there exists a $\tl{\chi}_1 \in \rC(X,R)$ such that $\pi_{X,R,K_1 \cup K_3}(\tl{\chi}_1) = \chi_1$ and $\n{\tl{\chi}_1}_{\rC(X,R)} < 2$. By $\pi_{X,R,K_1}(\tl{\chi}_1) = \pi_{K_1 \cup K_3,R,K_1}(\chi_1) = 0$, we have $\tl{\chi}_1 \in I_{X,R,K_1}$. Put $f_1 \coloneqq (f - f_0) \tl{\chi}_1 \in I_{X,R,K_1}$. For any $x \in K_3$, we have
\begin{eqnarray*}
(f - (f_0 + f_1))(x) = (f - f_0)(x) (1 - \tl{\chi}_1)(x) = (f - f_0)(x) (1 - \chi_1)(x) = 0.
\end{eqnarray*}
For any $x \in X \setminus K_3$, we have 
\begin{eqnarray*}
& & \n{(f - (f_0 + f_1))(x)}_R = \n{(f - f_0)(x)(1 - \tl{\chi}_1)(x)}_R \leq \n{(f - f_0)(x)}_R \n{(1 - \tl{\chi}_1)(x)}_R \\
& \leq & \n{(f - f_0)(x)}_R(\n{1}_R + \n{\tl{\chi}_1(x)}_R) < 3^{-1} \epsilon \cdot (1 + 2) = \epsilon.
\end{eqnarray*}
This implies $\n{f - (f_0 + f_1)}_{\rC(X,R)} \leq \epsilon$. Therefore $I_{X,R,K_0} + I_{X,R,K_1}$ is dense in $I_{X,R,K_0 \cap K_1}$.
\end{proof}

\begin{lmm}
\label{intersection}
Let $K_0$ and $K_1$ be closed subsets of $X$. If $\cD_3$ is an admissible triple, then the natural morphism
\begin{eqnarray*}
\rC(K_0,R) \wh{\otimes}_{\rC(X,R)}^{\bL} \rC(K_1,R) \to \rC(K_0,R) \wh{\otimes}_{\rC(X,R)} \rC(K_1,R)
\end{eqnarray*}
in $\Derba_{\cC}(\rC(X,R))$ is an isomorphism, and the morphism
\begin{eqnarray*}
\rC(K_0,R) \wh{\otimes}_{\rC(X,R)} \rC(K_1,R) \to \rC(K_0 \cap K_1,R)
\end{eqnarray*}
in $\Alg_{\cC}(\rC(X,R))$ associated to $\pi_{K_0,R,K_0 \cap K_1}$ and $\pi_{K_1,R,K_0 \cap K_1}$ is an isomorphism.
\end{lmm}

\begin{proof}
The assertion immediately follows from Proposition \ref{derived symmetric monoidal structure} (iii), Lemma \ref{subset operation and aritmetic operation}, and Lemma \ref{flatness of C(K,R)}.
\end{proof}

\begin{proof}[Proof of Theorem \ref{relation of cover by subsets}]
The equivalence between the conditions (i) and (ii) immediately follows from the definition of a $\cC$-cover. The equivalence between the conditions (iii) and (iv) follows from Proposition \ref{relation between derived covers and non-derived covers} and Lemma \ref{intersection}. It remains to show that (ii) is equivalent to (iii). Let $S_0$ be a finite subset of $S$. We denote by $T_{S_0}$ the (non-derived) Tate--\Cech complex
\begin{eqnarray*}
0 \to \rC(X,R) \to \prod_{(K,j) \in S_0} \rC(K,R) \to \prod_{(K_i,j_i)_{i=0}^{1} \in [S_0]^2} \rC(K_0,R) \wh{\otimes}_{\rC(X,R)} \rC(K_1,R) \to \cdots
\end{eqnarray*}
associated to $\Gamma^{\cD_3}_*(S_0)$. Put $X_0 \coloneqq \bigcup_{(K,j) \in S_0} j(K)$. By Theorem \ref{closed immersion with flat kernel}, it suffices to show that $T_{S_0}$ is strictly exact in $\Mod_{\cC}(R)$ if and only if $X_0 = X$.

\vspace{0.1in}
First, suppose $X_0 = X$. Then the exactness (without the strictness) immediately follows from Lemma \ref{intersection}, the exactness of the \Cech complex for (not necessarily continuous) functions, and the gluing lemma asserting that the openness of a subset of a topological space can be tested by the pullback by a finite closed cover. The strictness for the case where the norm of $R$ is trivial follows from the fact that the values of the norms of components of $T_{S_0}$ belong to $\ens{0,1}$. The strictness for the case where the norm of $R$ is non-trivial follows from open mapping theorem (cf.\ \cite{Bou53} Theorem I.3.3/1).

\vspace{0.1in}
Next, suppose that $X_0$ is a proper subset of $X$. Take an $x_1 \in X \setminus X_0$. By the finiteness of $S_0$, $X_0$ is closed. By the regularity of $X$, the characteristic function $\chi \colon X_0 \cup \ens{x_1} \to R$ of $\ens{x_1}$ is continuous. By Lemma \ref{Tietze for flat setting} applied to $K = X_0 \cup \ens{x_1}$, there exists a $\tl{\chi} \in \rC(X,R)$ such that $\pi_{X,R,X_0 \cup \ens{x_1}}(\tl{\chi}) = \chi$. We have $\tl{\chi}(x_1) = \chi(x_1) = 1$, and hence $\tl{\chi} \neq 0$, while we have $\pi_{X,R,j(K)}(\chi) = \pi_{X_0 \cup \ens{x_1},R,j(K)}(\chi) = 0$ for any $(K,j) \in S_0$. Therefore the map $\rC(X,R) \to \prod_{(K,j) \in S_0} \rC(K,R)$ in $T_{S_0}$ is not injective. This implies that $T_{S_0}$ is not strictly exact.
\end{proof}

We obtain an analogue of ``Kiehl's Theorem B'' (cf.\ \cite{Kie66} Hilfssatz 1.5) in rigid geometry.

\begin{crl}
\label{Theorem B}
Let $S$ be a finite set of closed immersions into $X$. If $\cD_3$ is an admissible triple, then the following are equivalent:
\begin{itemize}
\item[(i)] The equality $\bigcup_{(K,j) \in S} j(K) = X$ holds.
\item[(ii)] For any $M \in \Mod_{\cC}(\rC(X,R))$, the (non-derived) Tate--\Cech complex
\begin{eqnarray*}
0 \to M & \to & \prod_{(K,j) \in S} M \wh{\otimes}_{\rC(X,R)} \rC(K,R) \\
& \to & \prod_{(K_i,j_i)_{i=0}^{1} \in [S]^2} M \wh{\otimes}_{\rC(X,R)} \rC(K_0,R) \wh{\otimes}_{\rC(X,R)} \rC(K_1,R) \to \cdots
\end{eqnarray*}
of $M$ associated to $S$ is strictly exact in $\Mod_{\cC}(R)$.
\item[(iii)] For any $M \in \Derba_{\cC}(\rC(X,R))$, the total complex of the derived Tate--\Cech complex
\begin{eqnarray*}
0 \to M & \to & \prod_{(K,j) \in S} M \wh{\otimes}_{\rC(X,R)}^{\bL} \rC(K,R) \\
& \to & \prod_{(K_i,j_i)_{i=0}^{1} \in [S]^2} M \wh{\otimes}_{\rC(X,R)}^{\bL} \rC(K_0,R) \wh{\otimes}_{\rC(X,R)}^{\bL} \rC(K_1,R) \to \cdots
\end{eqnarray*}
of $M$ associated to $S$ is strictly exact in $\Mod_{\cC}(R)$.
\end{itemize}
\end{crl}

\begin{proof}
By Theorem \ref{relation of cover by subsets}, both of the conditions (ii) and (iii) applied to the case where $M$ is the regular $\rC(X,R)$-module object imply the condition (i). We show that the condition (i) implies the conditions (ii) and (iii). We denote by $C$ the \Cech complex
\begin{eqnarray*}
0 \to \rC(X,R) \to \prod_{(K,j) \in S} \rC(j(K),R) \to \prod_{(K_i,j_i)_{i=0}^{1} \in [S]^2} \rC(j_0(K_0) \cap j_1(K_1),R) \to \cdots
\end{eqnarray*}
associated to $(j(K))_{(K,j) \in S}$, which is naturally isomorphic in $\Chba_{\cC}(\rC(X,R))$ to (and hence will be identified with) the (non-derived) Tate--\Cech of $\rC(X,R)$ associated to $S$ by Lemma \ref{intersection}, and hence is strictly exact by the condition (i) and Theorem \ref{relation of cover by subsets}. By Lemma \ref{flatness of C(K,R)}, $C$ is termwise strongly flat. Therefore for any $M \in \Mod_{\cC}(\rC(X,R))$, $M \wh{\otimes}_{\rC(X,R)} C$ is isomorphic to $0$ in $\Derba_{\cC}(\rC(X,R))$ by Proposition \ref{derived symmetric monoidal structure} (ii), and hence is strictly exact by Proposition \ref{derived isom is qis} (iii). This implies the assertion (ii).

\vspace{0.1in}
Let $M \in \Derba_{\cC}(\rC(X,R))$. We denote by $C(M) \in \Chba(\Derba_{\cC}(\rC(X,R)))$ the derived Tate--\Cech complex of $M$ associated to $S$. The natural morphism $C(M) \to M \wh{\otimes}_{\rC(X,R)} C$ in $\Chba(\Derba_{\cC}(\rC(X,R)))$ is represented by a morphism in $\Chba(\Chba_{\cC}(\rC(X,R)))$ which is a termwise quasi-isomorphism by the termwise strong flatness of $C$ and Proposition \ref{derived symmetric monoidal structure} (iii), and the zero morphism $M \wh{\otimes}_{\rC(X,R)} C \to 0$ in $\Chba(\Chba_{\cC}(\rC(X,R)))$ is a termwise quasi-isomorphism by the assertion (ii). This implies that the zero morphism $\Tot(C(M)) \to 0$ in $\Chba_{\cC}(\rC(X,R))$ is a quasi-isomorphism by Proposition \ref{derived symmetric monoidal structure} (iv). Therefore $\Tot(C(M))$ is strictly exact by Proposition \ref{derived isom is qis} (iii). This implies the assertion (iii).
\end{proof}

\subsection{Acyclicity for a general topological cover}
\label{Acyclicity for a general topological cover}

We next compare general $\cC$-covers and derived covers of the Banach algebra $\rC(X,R)$ in the case when $\cD_3 = (\cC, X, R)$ forms an admissible triple. For this purpose, we will use the following proposition:

\begin{prp}
\label{bijectivity for morphism classes}
If $\cD_3$ is an admissible triple, then for any pair $(Y,\phi)$ of a compact Hausdorff topological space $Y$ and a morphism $\phi \colon \rC(X,R) \to \rC(Y,R)$ in $\Alg_{\cC}(R)$, there exists a unique continuous map $\Phi \colon Y \to X$ such that $\phi$ coincides with the pre-composition with $\Phi$.
\end{prp}

\begin{proof}
If $\cC = \BAb$, then the assertion immediately follows from \GelfandNaimark theorem. Suppose $\cC = \NBAb$. By \cite{Mih14} Corollary 2.3, we may replace the functor $\beta_{\cC} \colon \CH \to \TDCH$ by the composition of the functor $\rC(\cdot,R) \colon \CH \to \Alg_{\cC}(R)$ and the Berkovich spectrum functor $\cM_R \colon \Alg_{\cC}(R) \to \CH$, and we can replace $\iota_{\cC}$ by the natural transformation $\iota_R$ given by the evaluation map. Since $\cD_3$ is an admissible triple, $X$ is totally disconnected. This implies that $\iota_R(X)$ is a homeomorphism. Put $\Phi \coloneqq \iota_R(X)^{-1} \circ \cM_R(\phi) \circ \iota_R(Y)$. Then $\Phi$ is continuous by the continuity of $\iota_R(X)^{-1}$, $\cM_R(\phi)$, and $\iota_R(Y)$.

\vspace{0.1in}
Let $f \in \rC(X,R)$. For any $y \in Y$, we have
\begin{eqnarray*}
& & \phi(f)(y) = \iota_R(Y)(y)(\phi(f)) = \cM_R(\phi)(\iota_R(Y)(y))(f) = (\cM_R(\phi) \circ \iota_R(Y))(y)(f) \\
& = & (\iota_R(X) \circ \Phi)(y)(f) = \iota_R(X)(\Phi(y))(f) = f(\Phi(y))
\end{eqnarray*}
by the definition of $\iota_R$. This implies $\phi(f) = f \circ \Phi$. Therefore $\phi$ coincides with the pre-composition with $\Phi$. This completes the proof of the existence of a desired $\Phi$. We show its uniqueness. Let $\Phi'$ be a continuous map $Y \to X$ such that $\phi$ coincides with the pre-composition with $\Phi'$. Assume $\Phi \neq \Phi'$. Take a $y_0 \in Y$ with $\Phi(y_0) \neq \Phi'(y_0)$. By the Hausdorffness of $X$, the characteristic function $\chi \colon \ens{\Phi(y_0),\Phi'(y_0)} \to R$ of $\ens{\Phi'(y_0)}$ is continuous. By Lemma \ref{Tietze for flat setting}, there exists a $\tl{\chi} \in \rC(X,R)$ such that $\pi_{X,R,\ens{\Phi(y_0),\Phi'(y_0)}}(\tl{\chi}) = \chi$. We have
\begin{eqnarray*}
\chi(\Phi(y_0)) = \tl{\chi}(\Phi(y_0)) = (\tl{\chi} \circ \Phi)(y_0) = \phi(\tl{\chi})(y) = (\tl{\chi} \circ \Phi')(y_0) = \tl{\chi}(\Phi'(y_0)) = \chi(\Phi'(y_0)),
\end{eqnarray*}
but this contradicts $\chi(\Phi(y_0)) = 0 \neq 1 = \chi(\Phi'(y_0))$. This implies $\Phi = \Phi'$.
\end{proof}

Suppose that $\cD_2$ is an admissible pair. Put $\wh{\cD}_3 \coloneqq (\cC,\beta_{\cC}(X),R)$. Then $\wh{\cD}_3$ is an admissible triple, and hence Proposition \ref{bijectivity for morphism classes} is applicable to $\wh{\cD}_3$. For an object $A = (\rC(Y,R),\phi)$ of $C^*_{(\BAb,X,R)}$, we denote by $\Gamma_{\cD_3}^*(\phi)$ the continuous map $\Phi \colon Y \to \beta_{\cC}(X)$ such that $\phi \circ C^*_{\wh{\cD}_3}(\iota_{\cC}(X))$ coincides with the pre-composition with $\Phi$, which exists and is unique by Proposition \ref{bijectivity for morphism classes} applied to $\wh{\cD}_3$, and put $\Gamma_{\cD_3}^*(A) \coloneqq (Y,\Gamma_{\cD_3}^*(\phi)) \in \CH_{\cC}/\beta_{\cC}(X)$. For a subset $T \subset C^*_{(\BAb,X,R)}$, we put $\Gamma_{\cD_3}^*(T) \coloneqq \set{\Gamma_{\cD_3}^*(A)}{A \in T} \subset \CH_{\cC}/\beta_{\cC}(X)$. For a subset $S \subset \CH/X$, we put $\iota_{\cC} \circ S \coloneqq \set{(Y,\iota_{\cC}(X) \circ \Phi)}{(Y,\Phi) \in S} \subset \CH/\beta_{\cC}(X)$ and $\beta_{\cC}(S) \coloneqq \set{(\beta_{\cC}(Y),\beta_{\cC}(\Phi))}{(Y,\Phi) \in S} \subset \CH_{\cC}/\beta_{\cC}(X)$.

\begin{thm}
\label{relation of cover}
Let $(\Theta,\Delta)$ be either $(\CH_{\cC},\cD_3)$ or $(\CH,(\BAb,X,R))$. If $\cD_2$ is an admissible pair, then the following hold:
\begin{itemize}
\item[(i)] For any subset $S \subset \Theta/X$, the following are equivalent:
  \begin{itemize}
  \item[(i-i)] The set $\Gamma^{\Delta}_*(S)$ is a non-derived cover of $\rC(X,R)$ in $\cC$.
  \item[(i-ii)] The set $\Gamma^{\Delta}_*(S)$ is a derived cover of $\rC(X,R)$ in $\cC$.
  \item[(i-iii)] The set $\iota_{\cC}(X) \circ S$ is a $\cC$-cover of $\beta_{\cC}(X)$.
  \item[(i-iv)] The set $\beta_{\cC}(S)$ is a $\cC$-cover of $\beta_{\cC}(X)$.
  \item[(i-v)] The set $S$ is a $\cC$-cover of $X$.
  \end{itemize}
\item[(ii)] For any subset $T \subset C^*_{\Delta}$, the following are equivalent:
  \begin{itemize}
  \item[(ii-i)] The set $T$ is a non-derived cover of $\rC(X,R)$ in $\cC$.
  \item[(ii-ii)] The set $T$ is a derived cover of $\rC(X,R)$ in $\cC$.
  \item[(ii-iii)] The set $\Gamma_{\Delta}^*(T)$ is a $\cC$-cover of $\beta_{\cC}(X)$.
  \end{itemize}
\end{itemize}
\end{thm}

In order to show Theorem \ref{relation of cover}, we prepare several lemmata.

\begin{lmm}
\label{reduction to NBAb}
Let $Y$ be a compact Hausdorff topological space. If $\cD_2$ is an admissible pair, then the following hold:
\begin{itemize}
\item[(i)] The map $\iota_{\cC}(Y)$ is a surjective $\cC$-embedding.
\item[(ii)] The composition $\rC(\beta_{\cC}(Y),R) \to \rC(Y,R)$ with $\iota_{\cC}(Y)$ is an isometric isomorphism.
\end{itemize}
\end{lmm}

\begin{proof}
The claim for $\cC = \BAb$ is obvious. Suppose $\cC = \NBAb$. The assertion (i) follows from the construction of $\beta_{\cC}$ in \cite{Mih14} Lemma 1.8 and the compactness of $Y$, and the assertion (ii) follows from \cite{Mih14} Corollary 2.2 and \cite{Mih14} Corollary 3.6 (iii).
\end{proof}

We recall that at the beginning of this section, we defined the functor $\Gamma^{\cD^3} \colon \CH_{\cC}/X \to C^*_{\cD_3}$ assigning to each $(Y,\Phi) \in \CH_{\cC}/X$ the Banach ring $\rC(Y,R)$ equipped with the pre-composition $\rC(X,R) \to \rC(Y,R)$ with $\Phi$. We give a partial extension of \GelfandNaimark theorem.

\begin{lmm}
\label{Gel'fand--Naimark}
If $\cD_3$ is an admissible triple, then the functor $\Gamma^{\cD_3}$ gives a categorical equivalence between $(\CH_{\cC}/X)^{\op}$ and the essential image of the inclusion $C^*_{\cD_3} \hookrightarrow \Alg_{\cC}(\rC(X,R))$, and the functor represented by $R$ gives its quasi-inverse.
\end{lmm}

\begin{proof}
If $\cC = \BAb$, the assertion immediately follows from \GelfandNaimark theorem. Suppose $\cC = \NBAb$. We denote by $\ol{C}^*_{\cD_3}$ the essential image of the inclusion $C^*_{\cD_3} \hookrightarrow \Alg_{\cC}(\rC(X,R))$, and by $i \colon (\CH_{\cC}/X)^{\op} \to \ol{C}^*_{\cD_3}$ the functor given by $\Gamma^{\cD_3}$. Then $i$ is essentially surjective by the definition of the object class of $C^*_{\cC}$. For any morphism $\phi \colon \rC(Y_0,R) \to \rC(Y_1,R)$ in $C^*_{\cD_3}$, there exists a unique continuous map $\Phi \colon Y_1 \to Y_0$ such that $\Gamma^{\cD_3}(\Phi) = \Gamma^{(\cC,Y_0,R)}(\Phi) = \phi$ by Proposition \ref{bijectivity for morphism classes} applied to the admissible triple $(\cC,Y_0,R)$. Therefore $i$ is fully faithful. In particular, the restriction of the Berkovich spectrum functor $\Alg_{\cC_{\leq 1}}(R) \to \CH$ gives a quasi-inverse $\ol{C}^*_{\cD_3} \to \CH_{\cC}/X$ by \cite{Mih14} Corollary 2.2 and \cite{Mih14} Corollary 3.6 (iii), and is represented by $R \cong \rC(\ens{0},R)$ by \cite{Mih14} Corollary 3.6 (iv).
\end{proof}

\begin{lmm}
\label{equivalemce of homotopy epimorphism and strict epimorphism}
Let $(\rC(Y,R),\pi)$ be an object of $C^*_{\cD_3}$. If $\cD_2$ is an admissible pair, then the following are equivalent:
\begin{itemize}
\item[(i)] The morphism $\pi$ is a homotopy epimorphism in $\Alg_{\cC}(R)$.
\item[(ii)] The morphism $\pi$ is a strict epimorphism in $\Mod_{\cC}(R)$ and $\Mod_{\cC_{\leq 1}}(R)$.
\item[(iii)] The map $\Gamma_{\cD_3}^*(\pi)$ is a $\cC$-embedding.
\end{itemize}
\end{lmm}

\begin{proof}
By Lemma \ref{reduction to NBAb}, the claims are reduced to the case where $X$ and $Y$ are objects of $\CH_{\cC}$. In particular, $\cD_3$ is an admissible triple. The condition (iii) implies that $\Gamma_{\cD_3}^*(\pi)$ is a homeomorphism onto the image because $\iota_{\cC}(Y)$ is a homeomorphism, and hence implies the condition (ii) by Lemma \ref{Tietze for flat setting}. The condition (ii) implies the condition (i) by Theorem \ref{closed immersion with flat kernel}, because we have
\begin{eqnarray*}
& & \ker(\pi) = \set{f \in \rC(X,R)}{f \circ \Gamma_{\cD_3}^*(\pi) = 0} = \set{f \in \rC(X,R)}{\forall y \in Y, f(\Gamma_{\cD_3}^*(\pi)(y)) = 0} \\
& = & \set{f \in \rC(X,R)}{f |_{\Gamma_{\cD_3}^*(\pi)(Y)} = 0} = I_{X,R,\Gamma_{\cD_3}^*(\pi)(Y)}.
\end{eqnarray*}
Suppose that $\pi$ satisfies the condition (i). Then $\pi$ is an epimorphism in $\Alg_{\cC}(R)$ by Proposition \ref{homotopy epimorphism is epimorphism}, and hence the map $\Hom_{\Alg_{\cC}(R)}(\rC(Y,R),R) \to \Hom_{\Alg_{\cC}(R)}(\rC(X,R),R), \ \phi \mapsto \phi \circ \pi$ is injective. This implies the injectivity of $\Gamma_{\cD_3}^*(\pi)$ by Lemma \ref{Gel'fand--Naimark}. Since $X$ and $Y$ are compact Hausdorff topological spaces, $\Gamma_{\cD_3}^*(\pi)$ is a homeomorphism onto the image. Since $\pi$ coincides with the pre-composition with $\Gamma_{\cD_3}^*(\pi)$, it is a strict epimorphism in $\Mod_{\cC}(\rC(X,R))$ and $\Mod_{\cC_{\leq 1}}(\rC(X,R))$ by Lemma \ref{Tietze for flat setting}.
\end{proof}

\begin{proof}[Proof of Theorem \ref{relation of cover}]
For any subset $T \subset C^*_{\Delta}$, we have $T = \set{\phi \circ \Gamma^{\Delta}_*(\iota_{\cC}(X))^{-1}}{\phi \in \Gamma^{\Delta}_*(\Gamma_{\Delta}^*(T))}$. Therefore the assertion (ii) is reduced to the assertion (i). Since $\iota_{\cC}$ is a natural transformation $\id_{\CH} \Rightarrow \beta_{\cC}$, the condition (i-iii) is equivalent to the condition (i-iv) by Lemma \ref{reduction to NBAb} (ii). Since $\iota_{\cC}$ gives a natural isomorphism $\id_{\CH_{\cC}} \Rightarrow \beta_{\cC} |_{\CH_{\cC}}$, the condition (i-iv) is equivalent to the condition (i-v). By Lemma \ref{reduction to NBAb}, the equivalence of the conditions (i-i), (i-ii), and (i-v) for $(\CH,(\BAb,X,R))$ is reduced to the equivalence of them for $(\CH_{\cC},\cD_3)$. The equivalence of the conditions (i-i), (i-ii), and (i-v) for $(\CH_{\cC},\cD_3)$ immediately follows from Theorem \ref{relation of cover by subsets} and Lemma \ref{equivalemce of homotopy epimorphism and strict epimorphism}.
\end{proof}

\section{Application to derived and non-derived descent}
\label{Application to derived and non-derived descent}

In this last section, we show how the main results of this paper can be applied to the study of questions of descent for the Banach algebras of continuous functions. To explain what this means, we introduce some notions. Let $\cD_3 = (\cC, X, R)$ be an admissible triple. For a $(Y,\Phi)$ in $\CH/ X$, consider the monads $\cTnd(\Phi) = \Phi_* \Phi^*$ on $\Mod_{\cC}(\rC(X,R))$ and $\cTd(\Phi) = \Phi_* \bL \Phi^*$ on $\Derba_{\cC}(\rC(X,R))$ defined by the adjunctions
\begin{eqnarray*}
\Phi^* \colon \Mod_{\cC}(\rC(X,R)) & \leftrightarrows & \Mod_{\cC}(\rC(Y,R)) \colon \Phi_* \\
\bL \Phi^* \colon \Derba_{\cC}(\rC(X,R)) & \leftrightarrows & \Derba_{\cC}(\rC(Y,R)) \colon \Phi_*,
\end{eqnarray*}
where $\Phi^*$ is the scalar extension functor $(\cdot) \wh{\otimes}_{\rC(X,R)} \rC(Y, R)$, $\bL \Phi^*$ is the derived scalar extension functor $(\cdot) \wh{\otimes}_{\rC(X,R)}^{\bL} \rC(Y, R)$, and $\Phi_*$ is the restriction of scalars functor.

\vspace{0.1in}
Let $\cT$ denote $\cTnd(\Phi)$ (resp.\ $\cTd(\Phi)$), and abbreviate $\Mod_{\cC}$ (resp.\ $\Derba_{\cC}$) to $\rM$. We recall that the monad $\cT$ can be thought as a monoid over the monoidal category of endofunctors of $\rM(\rC(X,R))$, where the monoidal structure is given by composition of functors. Therefore, it makes sense to consider $\Mod(\cT)$, i.e.\ the category of left modules for the monad $\cT$. There is an adjunction
\begin{eqnarray*}
F_{\cT} \colon \rM(\rC(X,R)) & \leftrightarrows & \Mod(\cT) \colon U_{\cT},
\end{eqnarray*}
where $F_{\cT}$ is the free module object and $U_{\cT}$ is the forgetful functor. We note that $U_{\cT} \circ F_{\cT} = \cT$. Let us consider the comonad $L_{\cT} \coloneqq F_{\cT} \circ U_{\cT}$ and the category $\CoMod(L_{\cT})$ of left comodules over the comonad $L_{\cT}$. We denote $\CoMod(L_{\cT})$ by $\Desc_{\rM(\rC(X,R))}(\Phi)$, and call it the {\it non-derived} (resp.\ {\it derived}) {\it descent category associated to $\Phi$}. An object of $\Desc_{\rM(\rC(X,R))}(\Phi)$ can be described as a triple $(M, \rho, \sigma)$ of an $M \in \rM(\rC(X,R))$ and morphisms $\rho \colon \cT(M) \to M$ and $\sigma \colon M \to \cT(M)$ in $\rM(\rC(X,R))$ that give respectively the left action of $\cT$ and the left coaction of $L_{\cT}$ and that satisfy some further compatibility properties. See \cite{Bal12} Remark 1.4 for a detailed description of $\Desc_{\rM(\rC(X,R))}(\Phi)$, and \cite{Mes18} Theorem 3.1 for a preceding result on an abstract criterion of effective descent for $\Mod_{\cC_{\leq 1}}(A)$ for an $A \in \Alg_{\cC_{\leq 1}}(R)$ in the case where $R$ is $\R$ or $\C$, which implies that the Banach $A$-algebra $\prod_{e \in E} A/e A$ satisfies the effective descent for Banach modules with submultiplicative norm for any finite orthogonal system $E \subset A$ of idempotents of norm $\leq 1$ with $\sum_{e \in E} e = 1$ by a counterpart of Corollary \ref{clopen cover is derived cover} for $\cC_{\leq 1}$.

\vspace{0.1in}
We define the {\it comparison functor} $Q_{\cT}$ as the functor $\rM(\rC(X,R)) \to \Desc_{\rM(\rC(X,R))}(\Phi)$ assigning to each $M \in \rM(\rC(X,R))$ the tuple $(\cT(M), \epsilon_M, \cT(\eta_M))$, where $\epsilon$ is the counit $\cT \Rightarrow \id_{\rM(\rC(X,R))}$ of the adjunction and $\eta$ is the unit $\id_{\rM(\rC(X,R))} \Rightarrow \cT$ of the adjunction. We say that $\Phi$ {\it satisfies non-derived} (resp.\ {\it derived}) {\it effective descent} if the comparison functor $Q_{\cT}$ is an equivalence. We will apply the notion of the descent category in the following specific case:

\vspace{0.1in}
Let $S$ be a finite set of closed immersions into $X$. We put $Y \coloneqq \coprod_{(K,j) \in S} K$, and denote by $\Phi$ the coproduct $Y \to X$ of $(j)_{(K,j) \in S}$. If $S$ is a $\cC$-cover of $X$, then $\Phi$ is surjective by Theorem \ref{relation of cover by subsets}, and the category $\Desc_{\rM(\rC(X,R))}(\Phi)$ introduced so far is an abstract way of encoding the descent data given by a family $(M_{(K,j)})_{(K,j) \in S} \in \prod_{(K,j) \in S} \rM(\rC(K,R))$ with isomorphisms in $\rM(\rC(K_0,R) \wh{\otimes}_{\rC(X,R)} \rC(K_1,R))$ for all $(K_i,j_i)_{i = 0}^{1} \in S^2$ satisfying the cocycle condition. Therefore, in this case, the question of whether the comparison functor $Q_{\cT} \colon \rM(\rC(X,R)) \to \Desc_{\rM(\rC(X,R))}(\Phi)$ is an equivalence asks if it is possible to construct a unique (up to isomorphism) object of $\rM(\rC(X,R))$ out of a family of objects of $\rM(\rC(K,R))$ that are isomorphic on the intersections in a compatible way. In this situation, we say that $S$ satisfies {\it non-derived} (resp.\ {\it derived}) {\it effective descent} if the comparison functor $Q_{\cT} \colon \rM(\rC(X,R)) \to \Desc_{\rM(\rC(X,R))}(\Phi)$ is an equivalence. As an analogue of the equivalence between finite Banach modules over an affinoid algebra and coherent sheaves over the rigid analytic space associated to it (cf.\ \cite{Kie66} Theorem 1.2), we give our main result on derived and non-derived effective descent.

\begin{thm}
\label{derived effective descent}
The following are equivalent:
\begin{itemize}
\item[(i)] The equality $\bigcup_{(K,j) \in S} j(K) = X$ holds.
\item[(ii)] The functor $\bL \Phi^*$ is faithful.
\item[(iii)] The functor $\bL \Phi^*$ is conservative.
\item[(iv)] The functor $\Phi^*$ is faithful.
\item[(v)] The functor $\Phi^*$ is conservative.
\item[(vi)] The family $S$ satisfies derived effective descent.
\item[(vii)] The family $S$ satisfies non-derived effective descent.
\end{itemize}
\end{thm}

In order to show Theorem \ref{derived effective descent}, we recall some basic result of category theory that we need to prove Theorem \ref{derived effective descent}.

\begin{lmm}
\label{Derived category idempotent complete}
For any monoid $A$ of $\cC$, $\rM(A)$ is idempotent complete.
\end{lmm}

\begin{proof}
The assertion for $\Mod_{\cC}(A)$ follows from the fact that for any morphism $\pi \colon M \to M$ in $\Mod_{\cC}(A)$ with $\pi^2 = \pi$, the addition $\ker(\pi) \oplus \ker(1 - \pi) \to M$ is an isomorphism in $\Mod_{\cC}(A)$. The assertion for $\Derba_{\cC}(A)$ follows from \cite{BS01} Lemma 2.4, where the proof is given for the derived category of chain complexes bounded below of any exact category but the same proof can be given for the derived category of chain complexes bounded above.
\end{proof}

We recall that a morphism in a category is said to be a {\it bimorphism} if it is both a monomorphism and an epimorphism. A category is said to be {\it balanced} if every bimorphism is an isomorphism. The following is the main benefit to consider the triangulated category $\Derba_{\cC}(\rC(X,R))$ even when we are interested in the non-derived descent:

\begin{lmm}
\label{Triangulated balanced}
Every monomorphism in a triangulated category is split. In particular, every triangulated category is balanced.
\end{lmm}

\begin{proof}
The first assertion follows from \cite{Bal12} Remark 2.4. Since a split monomorphism is not an epimorphism unless the complement direct summand is an initial object, every bimorphism in a triangulated category is an isomorphism.
\end{proof}

\begin{lmm}
\label{Faithful functor from balanced category is conservative}
Let $F \colon C \to D$ be a faithful functor. If $C$ is balanced, then $F$ is conservative.
\end{lmm}

\begin{proof}
Suppose that $r$ is a morphism in $C$ such that $F(r)$ is an isomorphism in $D$. In particular, $F(r)$ is a monomorphism and an epimorphism. Since $F$ is faithful, it reflects both monomorphisms and epimorphisms. This implies that $r$ is a bimorphism, and hence is an isomorphism because $C$ is balanced.
\end{proof}

\begin{proof}[Proof of Theorem \ref{derived effective descent}]
By Lemma \ref{Derived category idempotent complete} and \cite{Bal12} Corollary 3.1, the condition (ii) is equivalent to the condition (vi). By Lemma \ref{Triangulated balanced} and Lemma \ref{Faithful functor from balanced category is conservative}, the condition (ii) implies the condition (iii), because $\Derba_{\cC}(\rC(X,R))$ is a triangulated category (cf.\ \cite{Sch99} Definition 1.2.18). The condition (vii) implies the faithfulness of $Q_{\cTnd}$ and hence the condition (iv).

\vspace{0.1in}
First, we show that the condition (ii) (resp.\ (iii)) implies the condition (iv) (resp.\ (v)). For a $Z \in \CH$, we denote by $J_Z \colon \Mod_{\cC}(\rC(Z,R)) \to \Derba_{\cC}(\rC(Z,R))$ the canonical embedding, which is fully faithful by \cite{Sch99} Corollary 1.2.28. Since $\rC(Y,R) \cong \prod_{(K,j) \in S} \rC(K,R)$ is a strongly flat object of $\Mod_{\cC}(\rC(X,R))$ by Lemma \ref{flatness of C(K,R)}, the correspondence assigning to each $M \in \Mod_{\cC}(\rC(X,R))$ the morphism $j_S(M) \colon M \wh{\otimes}_{\rC(X,R)}^{\bL} \rC(K,R) \to M \wh{\otimes}_{\rC(X,R)} \rC(K,R)$ in $\Derba_{\cC}(\rC(X,R))$ gives a natural isomorphism $j_S \colon \bL \Phi^* \circ J_X \Rightarrow J_Y \circ \Phi^*$ by Proposition \ref{derived symmetric monoidal structure} (iii). Let $f$ be a morphism in $\Mod_{\cC}(\rC(X,R))$ such that $\Phi^*(f)$ is a zero morphism (resp.\ an isomorphism) in $\Mod_{\cC}(\rC(Y,R))$. Then $(J_Y \circ \Phi^*)(f)$ is a zero morphism (resp.\ an isomorphism) in $\Der_{\cC}(\rC(Y,R))$, and hence so is $(\bL \Phi^* \circ J_X)(f)$ because $j_S$ is a natural isomorphism. Since $\bL \Phi^*$ is faithful (resp.\ conservative), $J_X(f)$ is a zero morphism (resp.\ an isomorphism) in $\Derba_{\cC}(\rC(X,R))$. Since $J_X$ is fully faithful, $f$ is a zero morphism (resp.\ an isomorphism) in $\Mod_{\cC}(\rC(X,R))$. This implies that $\Phi^*$ is faithful (resp.\ conservative).

\vspace{0.1in}
Secondly, we show that the condition (iv) (resp.\ (v)) implies the condition (i). For this purpose, it suffices to show that if $\Phi$ is not surjective, then $\Phi^*$ is not faithful (resp.\ conservative).  Suppose that an $x \in X$ is not contained in the image of $\Phi$. Let $r$ denote the identity morphism (resp.\ the zero morphism) $\rC(\ens{x},R) \to \rC(\ens{x},R)$ in $\Mod_{\cC}(\rC(X,R))$. By Lemma \ref{intersection}, we have $\Phi^*(\rC(\ens{x},R)) = \ens{0}$ and hence $\Phi^*(r)$ is a zero morphism (resp.\ an isomorphism), while $r$ itself is not a zero morphism (resp.\ an isomorphism). Therefore $\Phi^*$ is not faithful (resp.\ conservative).

\vspace{0.1in}
Thirdly, we show that the condition (i) implies the condition (ii). Since $\Derba_{\cC}(\rC(X,R))$ is an additive category and $\bL \Phi^*$ is an additive functor, it suffices to show that every morphism $f \colon M \to N$ in $\Derba_{\cC}(\rC(X,R))$ with $\bL \Phi^*(f) = 0$ is $0$ in order to show the faithfulness of $\bL \Phi^*$. We denote by $C \in \Chba_{\cC}(\rC(X,R))$ the (non-derived) Tate--\Cech complex of $\rC(X,R)$ associated to $S$ truncated at degree $0$, and by $\rho \colon \rC(X,R) \to C[-1]$ the canonical morphism in $\Chba_{\cC}(\rC(X,R))$. By the condition (i) and Theorem \ref{relation of cover by subsets}, $\rho$ is a quasi-isomorphism. By Lemma \ref{flatness of C(K,R)}, $C$ is termwise strongly flat. By $\bL \Phi^*(f) = 0$, we have $f \wh{\otimes}_{\rC(X,R)}^{\bL} C[-1] = 0$. By the commutativity of the diagram
\begin{eqnarray*}
\xymatrix{
M \ar[rr]^-{M \wh{\otimes}_{\rC(X,R)}^{\bL} \rho} \ar[d]^{f} & & M \wh{\otimes}_{\rC(X,R)}^{\bL} C[-1] \ar[d]^{f \wh{\otimes}_{\rC(X,R)}^{\bL} C[-1]} \\
N \ar[rr]^-{N \wh{\otimes}_{\rC(X,R)}^{\bL} \rho} & & N \wh{\otimes}_{\rC(X,R)}^{\bL} C[-1]
}
\end{eqnarray*}
in $\Derba_{\cC}(\rC(X,R))$, we have $(N \wh{\otimes}_{\rC(X,R)}^{\bL} \rho) \circ f = (f \wh{\otimes}_{\rC(X,R)}^{\bL} C[-1]) \circ (M \wh{\otimes}_{\rC(X,R)}^{\bL} \rho) = 0$. By Corollary \ref{Theorem B} (ii), $N \wh{\otimes}_{\rC(X,R)} \rho$ is a morphism in $\Chba(\Chba_{\cC}(\rC(X,R)))$ which is a termwise quasi-isomorphism. Therefore by Proposition \ref{derived symmetric monoidal structure} (iv), $N \wh{\otimes}_{\rC(X,R)}^{\bL} \rho$ is represented by a quasi-isomorphism. This implies that $N \wh{\otimes}_{\rC(X,R)}^{\bL} \rho$ is an isomorphism in $\Derba_{\cC}(\rC(X,R))$ and hence $f = (N \wh{\otimes}_{\rC(X,R)}^{\bL} \rho)^{-1} \circ (N \wh{\otimes}_{\rC(X,R)}^{\bL} \rho) \circ f = (N \wh{\otimes}_{\rC(X,R)}^{\bL} \rho)^{-1} \circ 0 = 0$.

\vspace{0.1in}
Finally, we show that the condition (i) implies the condition (vii). We note that we have already shown that the condition (i) is equivalent to the conditions (ii), (iii), (iv), (v), and (vi), and hence we can freely use them. The diagram
\begin{eqnarray*}
\xymatrix{
\Mod_{\cC}(\rC(X,R)) \ar[rr]^-{Q_{\cTnd}} \ar[d]^{J_X} & & \Desc_{\Mod_{\cC}(\rC(X,R))}(\Phi) \ar[d] \\
\Derba_{\cC}(\rC(X,R)) \ar[rr]^-{Q_{\cTd}} & & \Desc_{\Derba_{\cC}(\rC(X,R))}(\Phi)
}
\end{eqnarray*}
of functors commutes up to natural isomorphism, and the arrows other than $Q_{\cTnd}$ are fully faithful by the condition (vi). This implies that $Q_{\cTnd}$ is fully faithful.

\vspace{0.1in}
Therefore if $Q_{\cTnd}$ is essentially surjective, then the functor
\begin{eqnarray*}
\Sol_{\cTnd} \colon \Desc_{\Mod_{\cC}(\rC(X,R))}(\Phi) \to \Mod_{\cC}(\rC(X,R))
\end{eqnarray*}
assigning $\ker(c_0(M) - c_1(M))$ to each $M \in \Desc_{\Mod_{\cC}(\rC(X,R))}(\Phi)$ is a quasi-inverse of $Q_{\cTnd}$ by the condition (i) and Corollary \ref{Theorem B}, where $c_0(M)$ denotes the morphism
\begin{eqnarray*}
\prod_{(K,j) \in S} M_{(K,j)} & \to & \prod_{(K_i,j_i)_{i=0}^{1} \in S^2} M_{(K_0,j_0)} \wh{\otimes}_{\rC(X,R)} \rC(K_1,R) \\
(m_{(K,j)})_{(K,j) \in S} & \mapsto & (m_{(K_0,j_0)} \otimes 1)_{(K_i,j_i)_{i=0}^{1} \in S^2}
\end{eqnarray*}
in $\Mod_{\cC}(\rC(X,R))$, $c_1(M)$ denotes the morphism
\begin{eqnarray*}
\prod_{(K,j) \in S} M_{(K,j)} & \to & \prod_{(K_i,j_i)_{i=0}^{1} \in S^2} M_{(K_0,j_0)} \wh{\otimes}_{\rC(X,R)} \rC(K_1,R) \\
(m_{(K,j)})_{(K,j) \in S} & \mapsto & (\theta(M)_{(K_i,j_i)_{i=0}^{1}}(m_{(K_1,j_1)} \otimes 1))_{(K_i,j_i)_{i=0}^{1} \in S^2}
\end{eqnarray*}
in $\Mod_{\cC}(\rC(X,R))$, and $(M_{(K,j)})_{(K,j) \in S} \in \prod_{(K,j) \in S} \Mod_{\cC}(\rC(K,R))$ denotes the descent data associated to $M$ with compatible system $\theta(M) = (\theta(M)_{(K_i,j_i)_{i=0}^{1}})_{(K_i,j_i)_{i=0}^{1} \in S^2}$ of isomorphisms
\begin{eqnarray*}
\theta(M)_{(K_i,j_i)_{i=0}^{1}} \colon M_{(K_0,j_0)} \wh{\otimes}_{\rC(X,R)} \rC(K_1,R) \to M_{(K_1,j_1)} \wh{\otimes}_{\rC(X,R)} \rC(K_0,R)
\end{eqnarray*}
in $\Mod_{\cC}(\rC(X,R))$.

\vspace{0.1in}
We denote by $\LH_{\cC}(\rC(X,R))$ the full subcategory of $\Der_{\cC}(\rC(X,R))$ consisting of chain complexes concentrated at degrees in $\ens{-1,0}$ whose $(-1)$-st differential is a monomorphism. The essential image of $\LH_{\cC}(\rC(X,R))$ in $\Der_{\cC}(\rC(X,R))$ coincides with its left heart with respect to the left t-structure, and the correspondence assigning to each $N = (N_n,d_{N_n})_{n \in \Z} \in \Der_{\cC}(\rC(X,R))$ the object
\begin{eqnarray*}
\cdots \to 0 \to \coim(d_{N_{-1}}) \stackrel{D_N}{\to} \ker(d_{N_0}) \to 0 \to \cdots
\end{eqnarray*}
of $\LH_{\cC}(\rC(X,R))$ gives the cohomology functor $\LH^0 \colon \Der_{\cC}(\rC(X,R)) \to \LH_{\cC}(\rC(X,R))$ of the left t-structure by \cite{Sch99} Proposition 1.2.19, where $D_N$ denotes the canonical monomorphism $\coim(d_{N_{-1}}) \to \ker(d_{N_0})$ in $\Mod_{\cC}(\rC(X,R))$.

\vspace{0.1in}
By the condition (vi), there is a pair $(\tl{M},\psi)$ of an $\tl{M} = (\tl{M}_n,d_n)_{n \in \Z} \in \Derba_{\cC}(\rC(X,R))$ and an isomorphism $\psi_{\tl{M}} \colon \bL \Phi^*(\tl{M}) \to \prod_{(K,j) \in S} M_{(K,j)}$ in $\Derba_{\cC}(\rC(Y,R))$ compatible with the isomorphisms $\theta(M)$. We denote by $\Phi^*(\LH^0(\tl{M}))$ the chain complex
\begin{eqnarray*}
\cdots \to 0 \to \Phi^*(\coim(d_{-1})) \stackrel{\Phi^*(D_{\tl{M}})}{\longrightarrow} \Phi^*(\ker(d_0)) \to 0 \to \cdots
\end{eqnarray*}
of $\Mod_{\cC}(\rC(X,R))$, and by $\eta = (\eta_n)_{n \in \Z}$ the unit $\LH^0(\tl{M}) \to \Phi^*(\LH^0(\tl{M}))$ of $\cTnd$. We show that $\eta$ represents a monomorphism in $\LH_{\cC}(\rC(X,R))$. Applying $\LH^0$ to the image in $\Der_{\cC}(\rC(X,R))$ of the distinguished triangle
\begin{eqnarray*}
\Cone(\eta)[-1] \to \LH^0(\tl{M}) \to \Phi^*(\LH^0(\tl{M})) \to \Cone(\eta)
\end{eqnarray*}
in $\Chba_{\cC}(\rC(X,R))$, we obtain an exact sequence
\begin{eqnarray*}
\LH^0(\Cone(\eta)[-1]) \to \LH^0(\tl{M}) \stackrel{\eta}{\to} \Phi^*(\LH^0(\tl{M})) \to \LH^0(\Cone(\eta))
\end{eqnarray*}
in $\LH_{\cC}(\rC(X,R))$. Therefore it suffices to show that $\LH^0(\Cone(\eta)[-1])$ is strictly exact.

\vspace{0.1in}
Put $N \coloneqq \set{m \in \Phi^*(\coim(d_{-1}))}{\Phi^*(D_{\tl{M}})(m) \in \im(\eta_0)} \in \Mod_{\cC}(\rC(X,R))$. By the explicit presentation
\begin{eqnarray*}
\cdots \to 0 \to \coim(d_{-1}) & \stackrel{\binom{- D_{\tl{M}}}{\eta_{-1}}}{\longrightarrow} & \ker(d_0) \oplus \Phi^*(\coim(d_{-1})) \\
&  \stackrel{(\eta_0,\Phi^*(D_{\tl{M}}))}{\longrightarrow} & \Phi^*(\ker(d_0)) \to 0 \to \cdots
\end{eqnarray*}
of $\Cone(\eta)$, $\LH^0(\Cone(\eta)[-1])$ is given as
\begin{eqnarray*}
\cdots \to 0 \to \coim \left( \begin{array}{c} - D_{\tl{M}} \\ \eta_{-1} \end{array} \right) \stackrel{D_{\Cone(\eta)[-1]}}{\longrightarrow} \ker(\eta_0 \oplus \Phi^*(D_{\tl{M}})) \to 0 \to \cdots
\end{eqnarray*}
naturally identified with
\begin{eqnarray*}
\cdots \to 0 \to \coim(d_{-1}) \stackrel{\eta_{-1}}{\to} N \to 0 \to \cdots
\end{eqnarray*}
because $\eta_0$ is a strict monomorphism by the condition (i) and Corollary \ref{Theorem B}. By the commutativity of the diagram
\begin{eqnarray*}
\xymatrix{
\coim(d_{-1}) \ar[rrr]^-{\coim(d_{-1}) \wh{\otimes}_{\rC(X,R)} \rho} \ar[d]^{D_{\tl{M}}} & & & \coim(d_{-1}) \wh{\otimes}_{\rC(X,R)} C[-1] \ar[d]^{D_{\tl{M}} \wh{\otimes}_{\rC(X,R)} C[-1]} \\
\ker(d_0) \ar[rrr]^-{\ker(d_0) \wh{\otimes}_{\rC(X,R)} \rho} & & & \ker(d_0) \wh{\otimes}_{\rC(X,R)} C[-1]
}
\end{eqnarray*}
in $\Chba_{\cC}(\rC(X,R))$ whose horizontal arrows are quasi-isomorphisms again by the condition (i) and Corollary \ref{Theorem B}, we obtain $N = \Sol_{\cTnd}(Q_{\cTnd}(\coim(d_{-1})))$ and $\LH^0(\Cone(\eta)[-1])$ is strictly exact. This implies that $\eta$ represents a monomorphism in $\LH_{\cC}(\rC(X,R))$.

\vspace{0.1in}
Since $\rC(K,R)$ is a strongly flat object of $\Mod_{\cC}(\rC(X,R))$ and $M_{(K,j)}$ is concentrated at degree $0$ for any $(K,j) \in S$, $\psi_{\tl{M}}$ induces an isomorphism
\begin{eqnarray*}
\psi_{\LH^0(\tl{M})} \colon \Phi^*(\LH^0(\tl{M})) \to \prod_{(K,j) \in S} M_{(K,j)}
\end{eqnarray*}
in $\LH_{\cC}(\rC(X,R))$ compatible with $\theta(M)$. Since $\psi' \circ \eta \colon \LH^0(\tl{M}) \to \prod_{(K,j) \in S} M_{(K,j)}$ is a monomorphism in $\LH_{\cC}(\rC(X,R))$ and $\prod_{(K,j) \in S} M_{(K,j)}$ is concentrated at degree $0$, $\LH^0(\tl{M})$ belongs to the essential image of $\Mod_{\cC}(\rC(X,R))$ by \cite{Sch99} Proposition 1.2.29. This implies that $Q_{\cTnd}$ is essentially surjective.
\end{proof}

\vspace{0.3in}
\noindent {\Large \bf Acknowledgements}
\vspace{0.2in}

\noindent
The first listed author thanks Kobi Kremnizer for inspiring discussions on derived analytic geometry. During the preparation of this paper the first listed author has been supported by the DFG research grant BA 6560/1-1 ``\emph{Derived geometry and arithmetic}''.

\vspace{0.1in}
The second listed author thanks Fr\'ed\'eric Paugam for informing me of the interesting result in \cite{BK20} by the other author F.\ Bambozzi. We have started this joint work thanks to this information. The second listed author also thanks Takuma Hayashi for instructing him on elementary facts on derived categories, and colleagues in universities for daily discussions. The second listed author is also thankful to his family for their deep affection.


\end{document}